\documentclass[12pt,letterpaper]{article}
\usepackage{amsfonts}
\usepackage{amsmath,amssymb,color}
\usepackage{graphicx}

\newtheorem{theorem}{Theorem}

\newtheorem{lemma}{Lemma}

\newtheorem{proposition}{Proposition}
\newtheorem{remark}{Remark}

\newenvironment{proof}[1][Proof]{\noindent\textbf{#1.} }{\ \rule{0.5em}{0.5em}}

\newcommand{\Real}{\mathbb{R}}

\title{Extreme Analysis of a Non-convex and Nonlinear Functional of Gaussian Processes -- On the Tail Asymptotics of Random Ordinary Differential Equations}

\author{Jingchen Liu and Xiang Zhou}

\begin{document}
\maketitle

\begin{abstract}
	In this paper, we consider a  stochastic system described by a differential equation admitting a  spatially varying random coefficient.
	The differential equation  has been employed to model various  static physics systems such as elastic  deformation, water flow, electric-magnetic fields, temperature distribution, etc.
	A random coefficient is introduced to account for the system's uncertainty and/or imperfect measurements.
	This random coefficient is described by a Gaussian process (the input process) and thus the solution to the differential equation (under certain boundary conditions) is a complexed functional of the input Gaussian process.
	In this paper, we focus the analysis on the one-dimensional case and  derive   asymptotic approximations of the tail probabilities of the solution to the equation that has various physics interpretations under different contexts.
	This analysis rests on the literature of the extreme analysis of Gaussian processes (such as the tail approximations of the supremum) and extends the analysis to more complexed functionals.
\end{abstract}

\section{Introduction}\label{SecIntro}

Gaussian processes are often used to describe spatially varying uncertainties. In this paper, we consider the tail event of a non-linear and non-convex functional of a Gaussian process that arises naturally from the solution to a differential equation employed in various applications.
	Differential equation is a classic and powerful tool for the description of physics systems. Very often, microscopic heterogeneity or uncertainty of parameters exists such that the system cannot be completely characterized by a deterministic differential equation.
Stochastic models are usually employed, in combination with  differential equations, to account for such heterogeneity and/or uncertainty.
	In this paper, we are interested in one specific  differential equation that has applications to several subfields of physics and also has a close connection to stochastic differential equations.
We consider the following differential equation concerning a real-valued solution $v(x)$ defined on a $d$-dimensional compact subset  $\mathcal S \subset \Real^d$
\begin{equation}\label{eqn:EllipticPDE}
\nabla \cdot \{a(x) \nabla v(x) \}= p(x)\quad \mbox{for $x\in\mathcal S$}
\end{equation}
	where $a(x)$ and $p(x)$ are real-valued functions. The notation $\nabla v(x)$ is the gradient of $v(x)$ and $\nabla \cdot \{a(x) \nabla v(x) \}$ is the divergence of the vector field $a(x) \nabla v(x)$.  Appropriate boundary conditions will be imposed. The applications of this equation are discussed in Section \ref{SecApp}.

	Probability models may join the description of the system in multiple ways. In this paper, we adopt the formulation that the process $a(x)$ is a spatially varying stochastic process and thus the corresponding solution $v(x)$  itself (as a complexed function of $a(x)$) is also a stochastic process.
	In the applications we will discuss momentarily, the process $a(x)$ is physically constrained to be positive. 
	A natural modeling approach is  that $a(x)$ is a log-normal process, that is,
\begin{equation}
\label{eqn: log-normal-def}
a(x)= e^{-\sigma \xi(x)}
\end{equation}
where   $\xi(x)$ is a  Gaussian process living on $\mathcal S$, that is, for each finite subset $\{x_1,...,x_n\}\subset \mathcal S$, $(\xi(x_1),...,\xi(x_n))$ follows a multivariate normal distribution. The scalar $\sigma$ is  the noise amplitude  and is a positive constant.
We write $\log a(x) = -\sigma\xi(x)$ (instead of $\log a(x) =\sigma \xi(x)$) simply for notational convenience and it does not alter the problem.
	We are interested in developing sharp asymptotic approximations of the tail probabilities associated with $v(x)$, in particular, $P(\max_{x\in \mathcal S} |\nabla v(x)|> b)$ as $b\rightarrow \infty$. Such tail probabilities serve as stability measures of systems described by \eqref{eqn:EllipticPDE} in presence of uncertainties. Detailed discussions of the physics interpretation of $\nabla v(x)$ in different contexts as well as the connection to stochastic differential equations are provided in Section \ref{SecApp}.

In this paper, we restrict the analysis to the one-dimensional differential equation
\begin{equation}\label{eq}
(a(x)v'(x))'=p(x),\quad x\in [0,L]
\end{equation}
where $v'(x)$ denotes the derivative of $v(x)$ with respect to the spatial variable $x$.
The corresponding tail probability becomes
	\begin{equation}\label{tail}
	w(b) \triangleq P(\max_x |v'(x)|>b)\quad \mbox{as $b\rightarrow \infty$.}
	\end{equation}
	Under the Dirichlet boundary condition, $u(0) = u(L) =0$, and with representation \eqref{eqn: log-normal-def},   equation \eqref{eq} has a closed form solution
$v(x) = \int_0^x F(t)e^{\sigma\xi(t)}dt - \int_0^x e^{\sigma\xi(t)}dt\times \int_0^L F(s)e^{\sigma\xi(s)} ds/{\int_0^L e^{\sigma\xi(s)}  ds}$
where $F(x) \triangleq \int _0 ^x p(t) dt$
and its derivative is
\begin{equation} \label{eqn:Dirch-solution}
  v'(x) = e^{\sigma\xi(x)} \Big\{ F(x) -
\frac{\int_0^L F(t)e^{\sigma\xi(t)} dt}{\int_0^L e^{\sigma\xi(t)} dt} \Big\}.
\end{equation}
Therefore, $w(b)$ is the tail probability of a nonlinear and non-convex function of $\xi(x)$.

	The contribution of this paper is the derivation of a closed form sharp asymptotic approximations of $w(b)$ as $b\to \infty$. In particular, we discuss two situations: $p(x)$ is a constant and $|p(x)|$ admits one unique maximum in the interior of $[0,L]$.
	In addition to the asymptotic approximations of $w(b)$, this analysis also implies   qualitative descriptions of the most likely sample path along which $\max_x |v'(x)|$ achieves a high level. First, if $p(x)$ is a constant, then the maximum of $|v'(x)|$ is likely to be obtained at either end of the interval and it is unlikely to be obtained in the interior.
	Second, if $|p(x)|$ admits one unique interior maximum at $x_* = \arg\max_x|p(x)|$, then the maximum of $|v'(x)|$ is likely to be obtained at either of the three locations, $0$, $L$, or close to $x_*$, depending on the specific values of $p(0)$, $p(L)$, and $p(x_*)$. One notable feature  is that the maximizer of $|v'(x)|$ (in the elastic deformation application to be discussed in Section \ref{SecApp}, $v'(x)$ is the strain of a piece of material in presence of external force) is not necessarily obtained at $x_*$ where $|p(x)|$ (the external force) is maximized. 
A more detailed discussion on the most probable sample path given the rare event is provided in Section \ref{SecHeu}.

Upon considering $\max |v'(x)|$ as  a  functional of the input Gaussian process $\xi(x) $,
the current analysis sits well in the literature of rare-event analysis for Gaussian processes.
The technical development employs many tools in this literature.
A Gaussian process living on a general manifold is usually called a Gaussian random field.
The study of extremes of  Gaussian random fields focuses mostly on
the tail probabilities of the supremum of the field. The results contain
general bounds on $P(\max \xi(x)>b)$ as well as sharp asymptotic approximations as $%
b\rightarrow \infty$. A partial literature contains \cite%
{LS70,Bor03,LT91,Berman85}.
Several methods have been introduced to obtain bounds and asymptotic
approximations, each of which imposes different regularity conditions on the
random fields.
A short list of such methods includes the double sum method \cite{Pit95}, Euler--Poincar\'{e} Characteristics approximation (\cite{Adl81,TTA05,AdlTay07,TayAdl03}), the tube method (\cite{Sun93}), the Rice method (\cite{AW08,AW09}).
Recently, the exact tail approximation of integrals of exponential functions of Gaussian random fields is developed by \cite{Liu10,LiuXu11}. Efficient computations via importance sampling has been developed by \cite{ABL08,ABL09}. Recently, \cite{AST09} studied the geometric properties of high level excursion set for infinitely divisible non-Gaussian fields as well as the conditional distributions of such properties given the high excursion.

The analysis of high excursion of $|v'(x)|$ forms a challenging problem. Unlike the supremum norm and the integral of exponential functions, $\max_{x\in [0,L]} |v'(x)|$ as a function of $\xi (x)$ is neither sublinear nor convex and $v'(x)$ admits a much more complexed functional form than random variables studied in the existing literature.
In this paper, the analysis combines physics understanding, which helps with guessing  the most probable sample path of $\xi(x)$ given the high excursion of $|v'(x)|$, and  random field techniques  to derive approximations of $w(b)$. More technically, the development consists of analyzing the joint behavior of  $e^{\xi(x)}$ and  two integrals: $\int e^{\xi(x)}dx$ and $\int F(x) e^{\xi(x)}dx$. Approximations of the tail probabilities of $\max |v'(x)|$ are derived via the investigation of the joint extreme behaviors of these three quantities.

	The main reason that we constrain the analysis to the one-dimensional equation  is that the solution to \eqref{eq} can be written as a closed form function of $\xi(x)$. For the high-dimensional case, a well known fact is that the specific form of the solution to \eqref{eqn:EllipticPDE}  typically cannot be written as an explicit form of $\xi(x)$, which makes the theoretical analysis almost intractable.
	In the PDE literature, a popular solution to  \eqref{eqn:EllipticPDE}  is through numerical recipes \cite{Ghanem-book1991,Hou-Luo-WCE-2006}.
The accuracy  of  numerical analysis is typically designed for regular analysis. Rare-event analysis requires that the errors of the numerical methods vanish as the rarity parameter $b$ tends to infinity.
	Further adding to the difficulty, the numerical methods typically do not yield analytic relationship between $\xi(x)$ and $\nabla v(x)$, which is a crucial requirement for almost every theoretical analysis.
    Thus, in this paper, we focus on the one-dimensional analysis for which a closed form representation of $v(x)$ is available.
    Nonetheless, the one-dimensional analysis forms a necessary standpoint of the high-dimensional analysis. The  results derived in this paper also provide intuition and guideline of more general analysis for \eqref{eqn:EllipticPDE}.

The rest of this paper is organized as follows.
Section \ref{SecApp} is dedicated to the applications and connection to other probability literatures.
In Section \ref{SecMain}, we present the main results. The theorems are proved in Sections \ref{SecInH}. A supplemental material is provided to include more technical proofs of the  propositions and lemmas  supporting the proof of the main theorems.

\section{Connections and Applications}\label{SecApp}

\paragraph{Connection to an exit problem of  stochastic differential equations.}
The elliptic PDE \eqref{eqn:EllipticPDE} is closely connected an exit problem of  stochastic differential equations (SDE) and has a number of physics applications.
The discussions in this section are under the general multidimensional setting.
We first present its connection to SDE.
Using the representation of $a(x)$ in \eqref{eqn: log-normal-def}, we rewrite  equation \eqref{eqn:EllipticPDE} as
$\nabla^2 v(x) - \nabla \xi^\top (x) \nabla v(x) = e^{\xi(x)} p(x)$ and $v(x) = 0$ for $x\in \partial \mathcal S$, where $\nabla^2$ is the Laplace operator.
Define an operator $\mathcal A: C^2(\mathcal S)\to C(\mathcal S)$ as follows
$$\mathcal A v(x) \triangleq \nabla^2 v(x) - \nabla \xi^\top (x) \nabla v(x)$$
that is the generator of a continuous time process $X(t)$  taking values in $\mathcal S$ solving the SDE $d X(t) = -\nabla \xi(X(t)) + \sqrt 2 dW(t)$
where $W(t)$ is the $d$-dimensional standard Wiener process.
The function $\xi(x)$ is known as the potential function or energy landscape of $X(t)$.
Define $\tau = \inf\{t: X(t) \in \partial \mathcal S\}$ as the exit time out of the domain $\mathcal S$.
According to Dynkin's formula, the solution to \eqref{eqn:EllipticPDE} has the following representation $v(x) = -E\{\int_0^\tau e^{\xi(X(t))} p(X(t))dt | X(0)=x\}.$
Thus, the solution $v(x)$ is the expected integral from time $0$ up to $\tau$ of the process $X(t)$ that admits a random potential function. The realizations of $\xi(x)$ are usually of irregular shapes, which is an important feature for practical modeling. SDE's with irregular or rugged landscapes (in the current context, modeled as realizations of process $\xi(x)$) are considered in applications in chemical physics and biology. An incomplete list of references includes \cite{ANS00,BOSW95,HT03,MGR09}. The large deviations study of such processes are studied by \cite{DuSp12}. Therefore,  the current study naturally connects to the study of SDE's, though the main techniques are those in the Gaussian process literature.

\paragraph{Physics Applications.}

Equation \eqref{eqn:EllipticPDE}  is notably known in many disciplines
to describe time-independent physics problems.
Under different contexts, the solution $v(x)$ and  the coefficient $a(x)$ have their specific physics meanings. We list several such applications that admit precisely equation \eqref{eqn:EllipticPDE} to describe their systems.

In material science, the PDE \eqref{eqn:EllipticPDE} is known as the generalized Hook's law.
Consider a piece of material whose physical location is described by $\mathcal S$ with elasticity coefficient $a(x)$. Let $p(x)$ be the external force applied to the material at location $x\in \mathcal S$. Then, the deformation of the material (due to external force) at $x$ is given by the solution  $v(x)$ to equation \eqref{eqn:EllipticPDE}.
The elasticity coefficient $a(x)$ is determined mostly by the physical composition of the material. Its randomness is interpreted as follows. Consider that multiple pieces of material are made by the same manufacturer. Due to system noise and other  sources of errors, those pieces cannot be completely identical. This variation is described by the randomness of $a(x)$. In other words,  $a(x)$ is a random sample from a population of materials whose distribution is governed by the manufacturer.
The gradient $\nabla v(x)$ is interpreted as the strain of the material that breaks if $|\nabla v(x)|$ exceeds a certain threshold. Thus $P(\max_{x\in \mathcal S}|\nabla v(x)|> b)$ is the breakdown probability of the material in presence of external force $p(x)$ and it is an important risk measure in engineering.
In the following discussions, we will often refer to the material displacement application for the physics interpretations and intuitions.

In the study of  electrostatics, a piece of insulator, the shape of which is given by $\mathcal S$, has electrical  resistance $a(x)$. Then, the potential (or voltage) $v(x)$ solves equation \eqref{eqn:EllipticPDE} and $\nabla v(x)$ is the electric field.
This is known as the Gauss's law.
	The electrical resistance coefficient $a(x)$ may contain randomness based on a similar argument for the elasticity tensor.
	The high excursion of the electric field $\nabla v(x)$ induces insulation breakdown.
	Therefore, $P(\max_{x\in \mathcal S}|\nabla v(x)|> b)$ forms a risk measure  of the system.
	In groundwater hydraulics, the meaning of $v(x)$
is the hydraulic head (water level elevation), $a(x)$ is the hydraulic conductivity (or permeability), and $\nabla v$ indicates the water flow. This is known as the  Darcy's law.
	The elliptic PDE  \eqref{eqn:EllipticPDE} is also used to describe the steady-state distribution of heat where $v(x)$ carries the meaning of temperature at spatial location $x$ and the coefficient $a(x)$ is the heat conductivity of a thermal conductor whose physical location is given by $\mathcal S$. This is known as the Fourier's  law.

%


\section{Main results}\label{SecMain}

\subsection{The theorems} \label{SecThm}

We consider the differential equation \eqref{eq} with the Dirichlet condition. The gradient of the solution is given by \eqref{eqn:Dirch-solution}.
The random coefficient $a(x)$ takes the form \eqref{eqn: log-normal-def}, where $\xi(x)$ is a Gaussian process living on $[0,L]$.
To derive closed form approximations of $w(b)$, we list a set of technical conditions  concerning the input process $\xi(x)$ and $p(x)$.
\begin{itemize}
\item[A1] The process $\xi(x)$ is strongly stationary and furthemore $E\xi (x)=0$ and $E \xi^2(x) =1$.

\item[A2] The process $\xi (x)$ is almost surely three-time differentiable.
The covariance function admits the following expansion
$Cov(\xi (0),\xi (x))=C(x)=1-\frac{\Delta }{2}x^{2}+\frac{A}{24}%
x^{4}-Bx^{6}+o(x^{6}),$ as $x\to 0.$

\item[A3] For each $x$, $C(\lambda x)$ is a non-increasing function of $\lambda \in \Real^{+}$.

\item[A4] The function $p(x)$ is at least twice continuously
differentiable. In addition, it falls into either of the two cases.

\begin{itemize}

\item [Case 1.]
$|p(x)|$ admits its unique interior global maximum $x_* = \arg\max |p(x)|$ and $x_*\in (0,L)$. Furthermore, $|p(x)|$ is strongly concave (meaning that the second derivative is strictly negative) in a sufficiently small neighborhood around $x_*$.
\item [Case 2.]
$p(x)$ is constant.

\end{itemize}
\end{itemize}

	Assumption A2 is an important assumption for the entire analysis. In particular, three-time differentiability implies that the covariance function is at least six-time differentiable and the first, the third, and the fifth derivatives evaluated at the origin are all zero.
	The coefficients $\Delta$ and $A$ are known as the spectral moments that will be further discussed in the later analysis.
	Assumption A3 is imposed for technical purpose and it requires that $C(x)$ is decreasing on $\Real^+$ and increasing on $\Real^-$. Assumption A4 requires the uniqueness of the global maximum of $|p(x)|$. In case when $|p(x)|$ has more than one (interior) global maximum or the global maximum is at the boundary, the analysis can be adapted easily. This will be discussed in later remarks after the presentation of the asymptotic approximations.

Throughout our discussion we use the following notations for the asymptotic
behaviors. We say that $0\leq g(b)=O(h(b))$ if $g(b)\leq ch(b)$ for some
constant $c\in (0,\infty )$ and all $b\geq b_{0}>0$. Similarly, $g(b)=\Omega
(h(b))$ if $g(b)\geq ch(b)$ for all $b\geq b_{0}>0$. We also write $g(b)=\Theta (h(b))$ if $g(b)=O(h(b))$ and $g(b)=\Omega (h(b))$; $g(b)=o(h(b))$ as $b\rightarrow \infty $ if $g(b)/h(b)\rightarrow 0$ as $b\rightarrow \infty $; finally, $g(b) \sim h(b)$ if $g(b)/h(b) \to 1 $ as $b \to \infty$.


%
We present the asymptotic approximations of $w(b)$ under the two cases  in Assumption A4 respectively.
	We first consider the situation when $|p(x)|$ admits one unique maximum.
Let $x_{\ast}\triangleq \arg \max_{x\in [0,L]} |p(x)|$ be the unique interior maximum in $(0,L)$.
	Without loss of generality, we assume that $p(x_*)$, $p(0)$, and $p(L)$ are all positive. For the case that some or all of them are negative, the analysis is completely analogous. This will be mentioned in later remarks.

	The statement of the theorem needs the following notation.
We define three variables $u$, $u_0$ and $u_L$  that depend on the excursion level $b$. They are all approximately on the scale of $\frac{\log b}\sigma$.
For each $b>0$, let $u$ be the solution to%
\begin{equation}
p(x_{\ast})H(\gamma _{\ast }(u),u)e^{\sigma u}=b,  \label{form}
\end{equation}%
where   
\begin{equation}\label{HDef}
H(x,u)\triangleq |x|e^{-\frac12 {\Delta \sigma u}x^{2}}
\end{equation}
and $\gamma _{\ast }(u) \triangleq\arg\sup_{x>0} H(x,u)=u^{-1/2}\Delta ^{-1/2}\sigma ^{-1/2}$.
The identity \eqref{form} can be simplified to
\begin{equation}\label{form2}
\frac{p(x_{\ast })}{\sqrt{\sigma \Delta u}}e^{\sigma u-\frac{1}{2}}=b.
\end{equation}
	We introduce the notation  $\gamma _{\ast } (u)$ and  $H$
because they arise naturally in the derivation and have  geometric and probabilistic
interpretations that will be given in the proof of our main theorems.

For each $b>0$, let $u_{0}$ be the  solution to
\begin{equation}\label{u0}
\frac{e^{\sigma u_{0}}}{%
\sqrt{\Delta \sigma u_{0}}} \times \sup_{\{(x,\zeta ):x\leq \zeta \}}H_{0}(x,\zeta ; u_{0} )  =b.
\end{equation}
where
\begin{equation}\label{H0}
H_{0}(x,\zeta; u)\triangleq e^{-\frac{x^{2}}{2}} \times
E\Big[p(0)(x-Z)+\frac{p^{\prime}(0)}{2\sqrt{\Delta \sigma u}}(x-Z)^{2}   ~\Big | ~
Z\leq \zeta \Big].
\end{equation}
$Z$ is a standard Gaussian random variable
independent of any other random variables in the system; $E(\cdot |Z\leq \zeta )$ denotes the conditional expectation with respect to $Z$ given $Z\leq \zeta $.
	We provide further explanations of $H_0$.
	The second term inside the expectation  \eqref{H0} is $o(1)$. If we ignore it for the time being, then $H_0(x,\zeta;u)\approx p(0)e^{-x^2/2}[x- E(Z|Z\leq \zeta)].$
The last term in the definition of $H_0$ is important to obtain a sharp approximation of the tail probabilities.
	More properties of $H_0$ are included in Remark \ref{RemH} after the presentation of the theorem.

Similarly, we define $u_L$ by
\begin{equation}\label{ul}
\frac{e^{\sigma u_{L}}}{\sqrt{\Delta \sigma u_{L}}} \times \sup_{\{(x,\zeta ):x\leq \zeta \}}H_{L}(x,\zeta ; u_{L} )  =b.
\end{equation}
where
\begin{equation}\label{hl}
H_{L}(x,\zeta; u)\triangleq e^{-\frac{x^{2}}{2}} \times
E\Big[p(L)(x-Z)- \frac{p^{\prime}(L)}{2\sqrt{\Delta \sigma u}}(x-Z)^{2}   ~\Big | ~
Z\leq \zeta \Big].
\end{equation}
Note that the signs for the $p'$ terms in the definitions of $H_0$ and $H_L$ are different.

\begin{remark}
We now provide a remark on $u$, $u_0$, and $u_L$.
Note that $F(x)=\int_0^x p(t)dt$ is a bounded function and furthermore the factor,
$F(x) - {\int_0^L F(t)e^{\xi(t)} dt}/{\int_0^L e^{\xi(t)} dt}$,
is also bounded. In fact, this factor  converges to zero under the conditional distribution given the high excursion of $|v'(x)|$. Thus, if $|v'(x)|$ exhibits a high excursion, then $\xi(x)$ must also achieve a high level. The variable $u$ is interpreted as the level which $\xi(x)$ needs to achieve so that $|v'(x)|$  achieves the level $b$ around $x_*$.
The choice of $u$ also takes into account of this factor  that eventually vanishes.
A heuristic calculation of  $u$ will be provided in Section \ref{SecHeu}.
Similarly, $u_0$ and $u_L$ correspond to the high excursion levels of $\xi(x)$ at the two ends.
\end{remark}

We next introduce a number of constants/variables defined through the functions  $H_{0}$ and $H_L$.
They appear in the statement of the theorem.
For each $\zeta$, $u_0$, and $u_L$, maximizing $\log(|H_{0}|)$ and $\log (|H_L|)$ over $x\in (-\infty, \zeta]$
gives the definitions of the following functions 
\[ G_{0}(\zeta; u_0 )\triangleq \sup_{x\leq \zeta } \log |H_{0}(x,\zeta ;u_0)|, \quad G_{L}(\zeta; u_L )\triangleq \sup_{x\leq \zeta }\log |H_{L}(x,\zeta ;u_L)|.\]
 Define the maximizers of the $G$-function
\begin{eqnarray*}
\zeta _{0} &\triangleq &\arg \max_{\zeta }G_{0}(\zeta;u_0),\quad
\zeta _{L}\triangleq \arg \max_{\zeta } G_{L}(\zeta; u_L).
\end{eqnarray*}
Note that $\zeta_0$ depends on $u_0$ and $\zeta_L$ depends on $u_L$. To simplify the notation, we omit the indices $u_0$ and $u_L$ in the notation $\zeta_0$ and $\zeta_L$ when there is no ambiguity.
The second derivative of the $G$-functions evaluated at their maximizers are
\begin{eqnarray*}
\Xi _{0} &\triangleq &-\lim_{u_{0}\rightarrow \infty }\partial_\zeta^2 G_{0}|_{\zeta=\zeta_0,u=u_0},
\quad \Xi _{L}\triangleq -\lim_{u_{L}\rightarrow \infty}\partial_\zeta^2G_{L}|_{\zeta=\zeta_L, u=u_L}.
\end{eqnarray*}
Lastly, we define two constants
\begin{equation}
\label{kappa}
\begin{split}
\kappa _{0} \triangleq &\frac{A\zeta _{0}}{24\Delta ^{2}\sigma }-\frac{%
A\times E\left( Z^{4}|Z\leq \zeta _{0}\right) }{24\Delta ^{2}\sigma }+\frac{E[%
\frac{p^{\prime \prime }(0)}{6\sigma \Delta }(\zeta _{0}-Z)^{3}+\frac{Ap(0)}{%
24\Delta ^{2}\sigma ^{2}}Z^{4}(\zeta _{0}-Z)\left\vert Z\leq \zeta
_{0}\right. ]}{p(0)E(\zeta _{0}-Z\left\vert Z\leq \zeta _{0}\right. )},
 \\
\kappa _{L} \triangleq &\frac{A\zeta _{L}}{24\Delta ^{2}\sigma }-\frac{%
A\times E\left( Z^{4}|Z\leq \zeta _{L}\right) }{24\Delta ^{2}\sigma }+\frac{E[%
\frac{p^{\prime \prime }(L)}{6\sigma \Delta }(\zeta _{L}-Z)^{3}+\frac{Ap(L)}{%
24\Delta ^{2}\sigma ^{2}}Z^{4}(\zeta _{L}-Z)\left\vert Z\leq \zeta
_{L}\right. ]}{p(L)E(\zeta _{L}-Z\left\vert Z\leq \zeta _{L}\right. )},
\end{split}
\end{equation}%
where $Z$ is a standard Gaussian random variable and the constants  $A$ and  $\Delta$ are defined as in the expansion in Assumption A2.
The main results are summarized in the following theorems.

\begin{theorem}
\label{ThmMain} Suppose that $\xi (x)$ is a Gaussian process satisfying
conditions A1 - A3 and case 1 of A4. For all $x\in \lbrack 0,L]$, let $v'(x)$ be given as in \eqref{eqn:Dirch-solution}.
Let $u$, $u_0$, and $u_L$ be defined as in \eqref{form2}, \eqref{u0},  and \eqref{ul}.
If $p(x)$ is nonnegative at  $x = 0$, $x_*$, and $L$, then
\[
P(\sup_{x\in \lbrack 0,L]}|v^{\prime }(x)|>b)\sim
D\times u^{-1/2}e^{-u^{2}/2}+D_{0}\times u_{0}^{-1}e^{-u_{0}^{2}/2}+D_{L}\times u_{L}^{-1}e^{-u_{L}^{2}/2}
\]%
where $D$, $D_0$, and $D_L$ are constants defined as
\begin{eqnarray*}
D &=&\frac{\sqrt\Delta e^{\frac{A}{24\sigma
^{2}\Delta ^{2}}+\frac{p^{\prime \prime }(x_{\ast})}{6p(x_{\ast})\sigma
^{2}\Delta } }}{(2\pi )^{3/2}\sqrt{A-\Delta ^{2}}} 
\times
\\
&&~~ \int \exp \left \{-\frac{1}{2}
\left[\frac{\Delta ^{2}z^{2}}{A-\Delta ^{2}}-%
\frac{z}{\sigma }-\frac{y^{2}z}{\Delta }+\frac{A}{4\Delta^4 }y^{4}+y^{2}
\left(%
\frac{A}{2\sigma \Delta ^{3}}-\frac{p^{\prime \prime }(x_{\ast})}{p(x_{\ast
})\sigma \Delta ^{2}}
\right)\right]
\right\}dydz. \\
D_{0} &=&\frac{\sqrt\Delta e^{\kappa _{0}/\sigma }}{(2\pi )^{3/2}\sqrt{A-\Delta
^{2}}}\times \int \exp \Big\{-\frac{1}{2}\Big[\frac{\Delta ^{2}z^{2}}{A-\Delta ^{2}}
-\frac z \sigma
+\frac{\Xi _{0}}{\Delta }y^{2}\Big]\Big\}dydz \\
D_{L} &=&\frac{\sqrt\Delta e^{\kappa _{L}/\sigma }}{(2\pi )^{3/2}\sqrt{A-\Delta
^{2}}}\times \int \exp \Big\{-\frac{1}{2}\Big[\frac{\Delta ^{2}z^{2}}{A-\Delta ^{2}}
-\frac z \sigma
+\frac{\Xi _{L}}{\Delta }y^{2}\Big]\Big\}dydz.
\end{eqnarray*}
\end{theorem}

\begin{remark}
If $p(x)$ attains its maximum at multiple interior
points $x_1,\cdots, x_k$, then the approximation becomes
$P(\sup_{x\in \lbrack 0,L]}|v^{\prime }(x)|>b)\sim
\sum_{j=1}^{k}D(j)u^{-1/2}e^{-u^{2}/2}+D_{0}u_{0}^{-1}e^{-u_{0}^{2}/2}+D_{L}u_{L}^{-1}e^{-u_{L}^{2}/2},$
where $D(j)$'s are defined similarly as $D$ by replacing $x_\ast$ with $x_k$.
If the maximizer  $x_*$ is attained on the boundary, i.e.~$x_*=0$ or $L$, then the term $Du^{-1/2}e^{-u^{2}/2}$ should be
removed from the approximation.
\end{remark}

%
%

\begin{figure}[t]
\begin{center}
\includegraphics[scale=0.2]{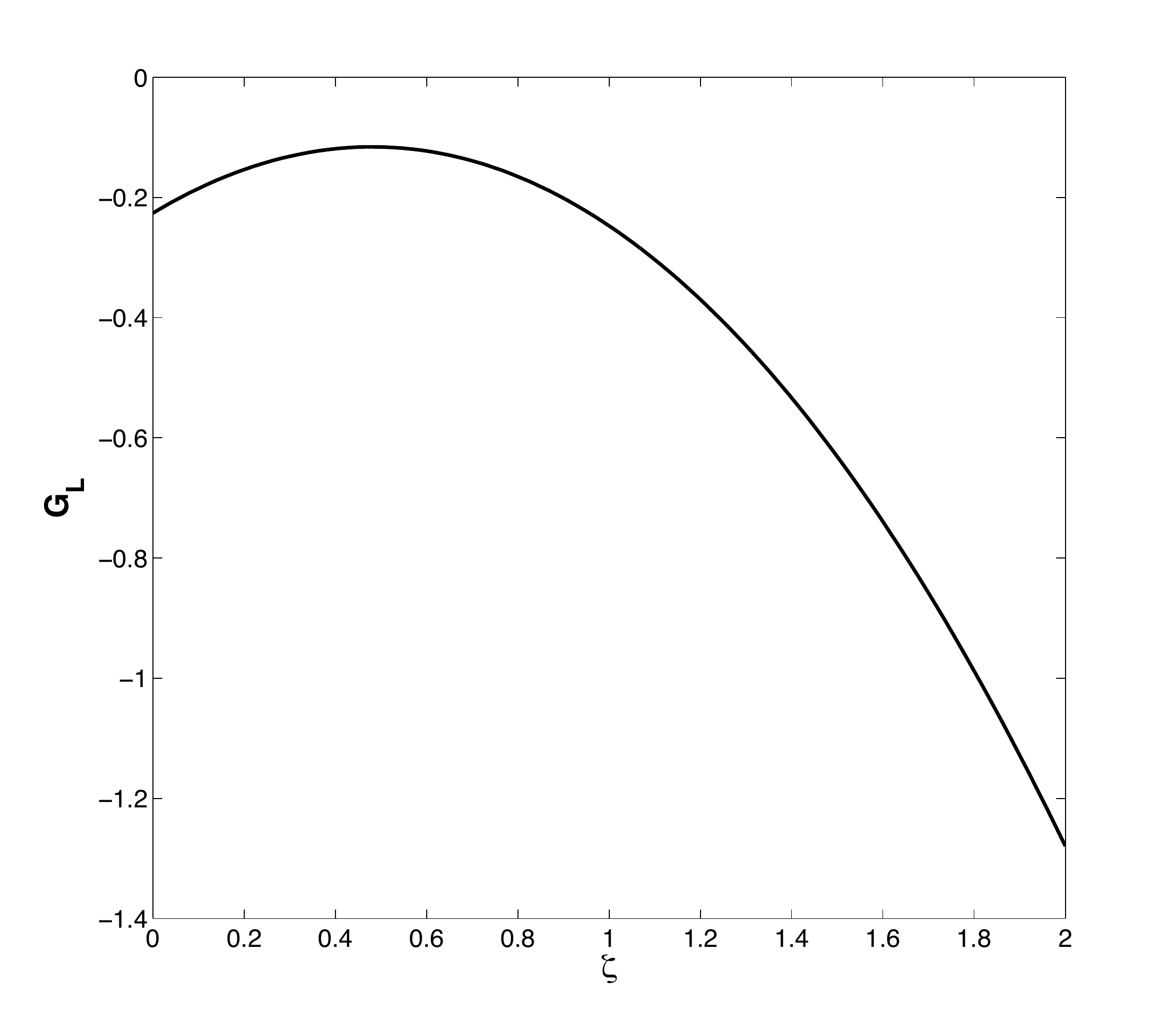}
\end{center}
\caption{Function $G_L(\zeta,u_L=\infty)$. }\label{G}
\end{figure}

\begin{remark}
\label{RemH}There are several features of the functions $H_0$ and $H_L$ that are
important in the analysis.
As $u_{L}\to \infty$, we  have that
$$H_L(x,\zeta;u_L)\to p(L)e^{-x^2/2}[x- E(Z|Z\leq \zeta)]>0 \quad \mbox{for all $x>0$}$$
and  $\zeta_{L}\approx 0.48$.
In addition, for $\zeta \leq 0.84$, we have 
$ \frac{\partial |H_L|}{\partial x}\vert _{(x,\zeta )=(\zeta ,\zeta)}>0,$
and thus $\max_{x\in (-\infty, \zeta]}\log |H_{L}(x,\zeta )|$ is solved at $x=\zeta $, that is, $G_L(\zeta; u_L) = \log |H_{L}(\zeta,\zeta;u_L )|.$
This calculation is important in the  technical derivations and it ensures that the maximum of $|v'(x)|$ is attained precisely at $x=L$ if $\max_{L-\varepsilon < x \leq L} |v'(x)| > b$.
To assist understanding, we numerically computed the function $G_L$ for $\zeta >0$ by setting $u_L = \infty$ and plot it in Figure \ref{G} for $p(L)=1$.
\end{remark}

\begin{remark}
The theorem assumes that $p(x)$ is positive at the important locations.
In the case when $p(x_*)<0$, we simply define $u$ through
$|p(x_{\ast})|e^{\sigma u+H(\gamma _{\ast }(u),u)}= b.$
The  definitions of other variables remain. Similarly, if $p(0)$ is negative we should generally define that
$H_{0}(x,\zeta; u)\triangleq \mbox{sign}(p(0)) e^{-\frac{x^{2}}{2}} \times
E\Big[p(0)(x-Z)+ \frac{p'(0)}{2\sqrt{\Delta \sigma u}}(x-Z)^{2}   ~ \Big| ~
Z\leq \zeta \Big],$
where ``sign" is the sign function. The same treatment can be applied to $H_L$ when $p(L)$ is negative.
The rest of the definitions remains.
To simplify the notation, we assume that $p(0)$ and $p(L)$ are positive and do not include the sign term.
\end{remark}

%
%

Now, we proceed to the approximation of $w(b)$ when $p(x)\equiv p_0>0$.
The approximation is very similar to Theorem \ref{ThmMain}, except that we do not have the term  $D\times u^{-1/2}e^{-u^{2}/2}$ and all the derivatives of $p(x)$ vanish.
To state the theorem, we need the following notation. We define a similar $H$-function and $G$-function as
$H_{h}(x,\zeta ) = p_{0}e^{-\frac{x^{2}}{2}}E(x-Z|Z\leq \zeta )$, and $G_h(\zeta) = \sup_{x\leq \zeta} \log |H_h(x,\zeta)|$
where the subscript ``$h$'' stands for this  constant force case $p(x)\equiv p_0$.
Furthermore, we define constants $\zeta _{h} =\arg \sup_{\zeta }G_{h}(\zeta ),$ $\Xi _{h} =-\partial_{\zeta}^2 G_{h}(\zeta _{h})$,
\begin{eqnarray*}
D_{h} &=&\frac{\Delta e^{\kappa _{h}/\sigma }}{(2\pi )^{3/2}\sqrt{A-\Delta
^{2}}}\times \int \exp \Big\{-\frac{1}{2}\Big[\frac{\Delta ^{2}z^{2}}{A-\Delta ^{2}}
-\frac z \sigma
+\frac{\Xi _{h}}{\Delta }y^{2}\Big]\Big\}dydz \\
\kappa _{h} &=&\frac{A\zeta _{h}^{4}}{24\Delta ^{2}\sigma }-\frac{AE\left(
Z^{4}|Z\leq \zeta _{h}\right) )}{24\Delta ^{2}\sigma }+\frac{AE\{Z^{4}(\zeta _{h}-Z)\left\vert Z\leq \zeta _{h}\right. \}}{{24\Delta
^{2}\sigma ^{2}}E(\zeta _{h}-Z\left\vert Z\leq \zeta _{h}\right. )}.
\end{eqnarray*}

\begin{theorem}
\label{ThmHomo} Suppose that the random field $\xi (x)$ satisfies the
Conditions A1-A3 and case 2 of A4. In addition, the external force $%
p(x)\equiv p_{0}$ is a positive constant.
For each $b>0$, let $u_{h}$ solve%
\[
\frac{e^{\sigma u_{h}}}{\sqrt{\Delta \sigma u_{h}}}\times \sup_{\{(x,\zeta): x\leq \zeta\}} H_h(x,\zeta)=b.
\]%
Then, we have the closed form approximation $P(\sup_{x\in \lbrack 0,L]}|v^{\prime }(x)|>b)\sim 2D_{h}e^{-u_{h}^{2}/2}.$

\end{theorem}

The proof of Theorem \ref{ThmHomo} is very similar to that of Theorem \ref{ThmMain}. We present it in the supplemental material.

%
%

\subsection{The intuitions behind the theorems and heuristic calculations}\label{SecHeu}

In this subsection, we provide intuitive interpretations of the previous  asymptotic approximations and further heuristic and non-rigorous calculations that help to understand the main proof. In particular, we focus mostly on the case when $p(x)$ is not a constant.

\paragraph{The most probable high excursion location.}
The approximation in Theorem \ref{ThmMain} consists of three pieces. The first
term $Du^{-1/2}e^{-u^{2}/2}$ corresponds to the probability that the maximum of $%
|v^{\prime }(x)|$ is attained close to the interior point $x_{\ast}=\arg \max_{x\in [0,L]} |p(x)|$;
the terms $D_{0}u_{0}^{-1}e^{-u_{0}^{2}/2}$
and  $D_{L}u_{L}^{-1}e^{-u_{L}^{2}/2}$  correspond to the probabilities that the
excursion of $|v'(x)|$ occurs at the two boundary points $x=0$ and  $x=L$, respectively. Thus, this three-term decomposition of $w(b)$ suggests that
${P(\max_{x\in [\varepsilon, x_*- \varepsilon]\cup [x_*+ \varepsilon,L- \varepsilon] } |v'(x)|>b~|~ \max_{[0,L]}|v'(x)|> b) }\rightarrow 0$ as $b\rightarrow \infty$ for any $\varepsilon>0$.
It is unlikely that the maximum is attained at a location other than the two ends or $x_*$.
In the context of the material failure problem, it suggests that, conditional on a failure, it is  likely that  the material breaks at the two ends or close to the place where the external force  is maximized.
As for which of the three locations is most likely to exhibit a high excursion, it depends on the specific functional forms of $p(x)$. Numerically, for each specific $b$, we can compute the three approximation terms in Theorem  \ref{ThmMain} and then compare among them. This provides the asymptotic probabilities of each of the high excursion locations.

We can further perform analytic calculations of the most likely high excursion location of $v'(x)$. Note that all the three terms decay exponentially fast with $u^2$, $u_0^2$, or $u_L^2$. Therefore, the smallest among $u$, $u_0$, and $u_L$ corresponds to the most likely location. Note that $u_0$ and $u_L$ take the same form. Thus, we only need to compare $|p(0)|$ and $|p(L)|$. The larger one corresponds to a smaller $u$-value and therefore yields a more likely high excursion. To compare the boundary case and the interior case, we need to compare $u$ and $u_0$ (or $u_L$). We take $u_0$ as an example. Note that both $u$ and $u_0$ are defined by $b$  implicitly through the equations in similar forms. Therefore, it is sufficient to compare among the two terms
\[
|p(x_*) e^{H(\gamma_*,u)}| =|p(x_*)| \frac{e^{-1/2}}{\sqrt{\sigma\Delta u}}, \mbox{ and  } \frac{\sup_{x\leq \zeta}H_0(x,\zeta,u_0)}{\sqrt{\sigma\Delta u_0}} \sim \frac{|p(0)|\sup_{x\leq \zeta} e^{-x^2/2}E(x-Z|Z\leq \zeta)}{\sqrt{\sigma\Delta u}} .
\]
Furthermore, we consider the ratio 
$$
r \triangleq
\frac{\sup_{x\leq \zeta} e^{-x^2/2}E(x-Z|Z\leq \zeta)}{\sqrt{\sigma\Delta u}}\Big / \frac{e^{-1/2}}{\sqrt{\sigma\Delta u}}
=\sup_{(\zeta,x), s.t.~x\leq \zeta} e^{\frac{1-x^2}{2}} E(x-Z|Z\leq \zeta).$$
 Note that $r$ is a universal constant strictly greater than $1$. If
 $|p(x_*)|> r|p(0)| ,$
 then $x_*$ is a more probable location to observe a high excursion; if
  $|p(x_*)| < r |p(0)|,$
  then zero is a more probable location. If $p(x)$ is a constant, then $u > u_0 = u_h$. This is why the maximum of $v'(x)$ is not attained in the interior for this case.

 \paragraph{Heuristic calculations.}
 In what follows, we provide a heuristic argument for \eqref{form} that defines $u$.
Note that a high level of $|v'(x)|$ implies a high level of $\xi(x)$.
Suppose that $\xi (x)$ attains its maximum at  $\tau\in[0,L]$
that is very close to $x_*$.
Then, the process $\xi (x)$ is   approximately quadratic near $\tau$. In particular, conditional on $\xi(\tau)  = u$, $\xi(x)$ admits the representation that
$\xi (x)= E(\xi(x)| \xi(\tau) = u)  + g(x - \tau),$
where  $g(x)$ is a mean-zero Gaussian process.
If we ignore the variation of $g(x)$, then according to Assumption A2 we have that
$\xi (x) \approx E(\xi(x)| \xi(\tau) = u) = uC(x-\tau) \approx u - \frac{\Delta u}{2 } (x- \tau)^2$.
Thus, $e^{\sigma \xi (t)}/\int_0^L e^{\sigma \xi(s)}ds$ is approximately
a Gaussian density with mean $\tau$ and variance $\Delta ^{-1}\sigma
^{-1}u^{-1}$ and the Laplace approximation can be followed. We then have the following  approximation
\[
 F(x)-\int_{0}^{L}\frac{F(t)e^{\sigma \xi (t)}}{\int_{0}^{L}e^{\sigma
\xi (s)}ds}dt \approx F(x)- F(\tau) \approx p(\tau)(x-\tau),
\]%
where $F(x)=\int_{0}^{x}p(s)ds$. Therefore, the strain is approximately
\[
v'(x) = e^{\sigma \xi (x)}\left( F(x)-\int_{0}^{L}\frac{F(s)e^{\sigma \xi (s)}}{%
\int_{0}^{L}e^{\sigma \xi (t)}dt}ds\right) \approx e^{\sigma u-\frac{\sigma \Delta u}{2}(x-\tau)^{2}}\times p(\tau)(x-\tau)
\]%
that is maximized when $x-\tau=\gamma _{\ast }(u)=(u\Delta \sigma)^{-1/2}$. If $\tau$ is close to $x_*$, we can replace $p(\tau)$ by $p(x_*)$ and further approximate $\max_x |v'(x)|$ by
\[
\max_x |v'(x)|\approx p(x_*)\gamma _{\ast }(u)e^{\sigma u-\frac{%
\sigma \Delta u}{2}\gamma _{\ast }^2(u)}=p(x_*)e^{\sigma u } H(\gamma_\ast(u),u).
\]%
If we let the above approximation equal $b$ then this is precisely how $u$ is defined in \eqref{form}.
Therefore, $u$ is the minimum level that the process needs to exceed so that $\max |v'(x)|$ could exceed the level $b$.
It is easier for $|v'(x)|$ to exceed a high level when $\tau$ is very closed to $x_*$.
If $\tau $ is distant from $x_{\ast}$, say $|\tau -x_{\ast}|>\varepsilon $, then the
approximation would be
$\max |v^{\prime }(x) | \approx p(\tau )\gamma _{\ast }(u)e^{\sigma u-\frac{\sigma
\Delta u}{2}\gamma _{\ast }^{2}(u)}.$
Since $p(x)$ is strongly concave around $x_{\ast}$, then $p(\tau )\approx
p(x_{\ast})+\frac{p^{\prime \prime }(x_{\ast})}{2}(\tau -x_{\ast
})^{2}<p(x_{\ast})-\lambda \varepsilon ^{2}$ for some $\lambda >0$. Thus,  it is necessary for $\xi(\tau)$ to achieve a higher level  than $u$ when $\tau$ is distant from $x_*$.

The above heuristic calculation outlines our analysis strategy for an interior point $x\in (0,L)$. For the boundary case, i.e., $\tau$ is $O(u^{-1/2})$ distance from $0$ or $L$, the calculations are different. Basically, if we write
$ E(F(S)) = \int_{0}^{L}F(s)e^{\sigma \xi (s)}ds/{\int_{0}^{L}e^{\sigma \xi (v)}dv},$
then $S$ follows approximately a normal distribution when $\tau\in [\varepsilon,L-\varepsilon]$ is in the interior. For the boundary case, e.g., $\tau = L- \zeta/\sqrt u$, the support of the random variable $S$ is truncated beyond the region $[0,L]$ and thus all the calculations consists of conditional normal distributions. This is how we define the functions  $H_0(x,\zeta,u)$ and $H_L(x,\zeta,u)$ that consist of expectations of conditional Gaussian distributions.

When the external force $p(x)$ is a constant, the asymptotic approximation only consists of two terms that correspond to the probabilities that the high excursion of $|v'(x)|$ occurs at either end of the interval.

\section{Proof of Theorem \protect\ref{ThmMain}} \label{SecInH}

The proof in Theorem \protect\ref{ThmMain} is based on the following  inclusion-exclusion formula
\begin{equation}
\sum_{i=1}^{3}P(E_{i})-\sum_{i=1}^{2}\sum_{j=i+1}^{3}P(E_{i}\cap E_{j})\leq
P(\max_{[0,L]}v^{\prime }(x)>b)=P( \cup _{i=1}^{3}E_{i})\leq
\sum_{i=1}^{3}P(E_{i}).  \label{bern}
\end{equation}%
where the  events $E_1, E_2, E_3$ are defined as follows 
\begin{eqnarray}
E_{1}&=&\left\{ \max_{x\in  [u^{-1/2+\delta },L-u^{-1/2+\delta }]} |v'(x)| \  >b
\right \}, \quad E_{2}=\left\{\max_{x\in [0,u^{-1/2+\delta }]}|v'(x)| \   >b\right\}, \notag\\
E_{3}&=&\left\{\max_{x\in \lbrack L-u^{-1/2+\delta },L]}|v^{\prime
}(x)|\  >b\right\},
 \label{event}
\end{eqnarray}
where  $\delta >0$ is sufficiently small but independent of $b$.

The main body is to  derive the approximations for $P(E_{i})$ and $P(E_{i}\cap E_{j})$.
Section \ref{sec:3.1} includes the derivations for  $P(E_1)$ and Section \ref{SecE3} includes the derivations for $P(E_2)$ and $P(E_3)$.
In addition, from the following detailed derivation of $P(E_{1})$ and $P(E_{3})$, it is straight forward to have that
\begin{equation}\label{minorint}
P(E_{1}\cap E_{2})+P(E_{1}\cap E_{3})+P(E_{2}\cap
E_{3})=o(P(E_{1})+P(E_{2})+P(E_{3})).
\end{equation}
Thus,  we with  complete   the proof of Theorem \protect\ref{ThmMain} by the inequality in (\ref{bern})

In the following analysis, we use both $x$ and $t$ to denote the spatial index. In particular, we use $t$ for the index when doing integration and use $x$ when taking the supremum. We first present the Borel-TIS lemma, which was proved independently by \cite{Bor75, CIS}.

\begin{lemma}[Borel-TIS]\label{LemBorel}
Let $\xi(x)$, $x\in \mathcal U$, $\mathcal U$ is a parameter set,
be mean zero Gaussian random field and  $\xi$ is almost surely
bounded on $\mathcal U$. Then,
$E(\max_{\mathcal U}\xi(x) )<\infty,$ and for any real number $b$
\[
P\left(\max_{x\in \mathcal{U}}\xi(  x)
-E[\max_{x\in\mathcal{U}}\xi\left( x\right)  ]\geq b\right)\leq e^{
-\frac{b^{2}}{2\sigma_{\mathcal{U}}^{2}}}  ,\quad \mbox{where $\sigma_{\mathcal{U}}^{2}=\max_{x\in \mathcal{U}}Var[\xi( x)]$.}
\]

\end{lemma}

\subsection{Approximation for $P(E_{1})$}
\label{sec:3.1}

 Consider the
following change of variables  from $(\xi(x_*),\xi'(x_*),\xi''(x_*)) $ to $(w,y,z)$  that depends on the variable $u$
\[
w\triangleq \xi (x_{\ast}) - u, \quad  y \triangleq \xi ^{\prime }(x_{\ast}),\quad
z\triangleq u+\xi ^{\prime\prime } (x_{\ast})/ \Delta.
\]
We further write
$P(\cdot | \xi (x_{\ast}) = u + w, \xi ^{\prime }(x_{\ast})= y, \xi''(x_*) = -\Delta ( u-z)) = P(\cdot | w,y,z)$ and obtain
\begin{equation}
P(E_{1})=\Delta \int P(E_{1}|w,y,z)h(w,y,z)dwdydz.  \label{intt}
\end{equation}%
where $h(w,y,z)$ is the density function of $(\xi (x_{\ast}),\xi ^{\prime
}(x_{\ast}),\xi ^{\prime \prime }(x_{\ast}))$ evaluated at $(u+w,y,-\Delta
(u-z))$. The following proposition localizes the event to a region  convenient for Taylor expansion on $\xi (x)$.

\begin{proposition}
\label{PropLocal} Under the conditions in Theorem \ref{ThmMain},  consider%
\[
\mathcal L_{u}=\{|w|<u^{3\delta }\}\cap \{|y|<u^{1/2+4\delta }\}\cap
\{|z|<u^{1/2+4\delta }\}.
\]%
Then, for any $\delta >0$, we have that
$P(\mathcal L_{u}^c;E_{1})=o(u^{-1}e^{-u^{2}/2}).$
\end{proposition}

The proof of this proposition is presented in the supplemental material.
This proposition localizes the event $E_1$ to a region  where the maximum of $v'(x)$ is achieved around $x_*$.
The above proposition suggests that we only need to consider the event on the set $\mathcal L_{u}$, that is,
$\Delta \int_{\mathcal  L_{u}}P(E_1|w,y,z)h(w,y,z)dwdydz.$

Conditional on $(\xi (x_{\ast}),\xi ^{\prime }(x_{\ast}),\xi ^{\prime
\prime }(x_{\ast}))$,
we write the process in the following
representation
$\xi (x) =E(\xi (x)|w,y,z)+g(x-x_{\ast}).$
The process $g(x-x_*)$ represents the variation of $\xi(x)$ when $\xi(x_*)$ and its first two derivatives have been fixed.
Thus, $g(x-x_*)$ is a mean-zero Gaussian process almost surely three-time differentiable. Using conditional Gaussian calculations and Taylor expansion, we have that $Var(g(x-x_*))=O(|x-x_*|^{6})$, that is, $g(x-x_*)=O_{p}(|x-x_*|^{3})$ as $g$ is the remainder term after conditioning on $\xi(x_*)$ and the first two derivatives.
Note that the distribution of $g(x)$ is free of $(w,y,z)$.
Let $\bar E(x;w,y,z) \triangleq E(\xi (x)|w,y,z)$.
By means of the conditional Gaussian calculations  (Chapter 5.5 \cite{AdlTay07}), we have that
\begin{equation*}
\begin{split}
\partial\bar E(x_*;w,y,z) = y,
~~ \partial^2\bar E(x_*;w,y,z) =  - \Delta (u-z),\\
~~ \partial^3\bar E(x_*;w,y,z) = -\frac{A}{\Delta}y,
~~ \partial^4\bar E(x_*;w,y,z) = Au+O(z),
\end{split}
\end{equation*}
where ``$\partial$'' is the partial derivative with respect to $x$.
We perform Taylor expansion on $\bar E(x;w,y,z)$. Using the notation   $\vartheta(x)=O(u^{1/2+4\delta }x^{4}+ux^{6})$, we  obtain that on the set $\mathcal L_u$
\begin{equation}
\begin{split}
\label{expansion}
\xi(x)=&u+w+y(x-x_{\ast})-\frac{\Delta (u-z)}{2}(x-x_{\ast})^{2}  \\
&~~~  -\frac{A}{6\Delta }y(x-x_{\ast})^{3}+\frac{Au}{24}(x-x_{\ast
})^{4}+g(x-x_{\ast})+\vartheta (x-x_{\ast})   \\
=&u+w+\frac{y^{2}}{2\Delta (u-z)}-\frac{\Delta (u-z)}{2}\Big( x-x_{\ast}-%
\frac{y}{\Delta (u-z)}\Big) ^{2}  \\
&~~~-\frac{A}{6\Delta }y(x-x_{\ast})^{3}+\frac{Au}{24}(x-x_{\ast})^{4}+g(x-x_{\ast})+\vartheta (x-x_{\ast}).  
\end{split}
\end{equation}%
For $\delta >0$,
we further localize the event by the following proposition, the proof of which is provided in the supplemental material.

\begin{proposition}
\label{PropG} For each $\delta ,\delta ^{\prime }>0$ chosen small enough and
$\delta ^{\prime }>24\delta $, we have that
\begin{eqnarray*}
P\Big(
\sup_{|x|>u^{-1/2+8\delta }}(|g(x)|-\delta ^{\prime }ux^{2})
>0,\ \mathcal L_{u}\Big ) &=& o(u^{-1}e^{-u^{2}/2}), \\
P\Big(\sup_{|x|\leq u^{-1/2+8\delta }} |g(x)| >u^{-1/2+\delta ^{\prime}},~ \mathcal L_{u}\Big )&=&o(u^{-1}e^{-u^{2}/2}).
\end{eqnarray*}
\end{proposition}
With this proposition, let
\[
\mathcal  L_{u}^{\prime }=\mathcal L_{u}\cap
\Big\{\sup_{|x|>u^{-1/2+8\delta }}[|g(x)|-\delta^{\prime }ux^{2}]<0\Big\}
\cap
\Big \{\sup_{|x|\leq u^{-1/2+8\delta }}|g(x)|<u^{-1/2+\delta ^{\prime  }} \Big\}.
\]%
We further reduce the event to
$\Delta \int_{\mathcal L_{u}}P(E_{1},\mathcal L_u '|w,y,z)h(w,y,z)dwdydz.$
The analysis of $P(E_{1})$ consists of three steps.
\begin{enumerate}
\item [Step 1] We continue the
calculation in (\ref{expansion}) and write  $v'(x)$ in an
analytic form of $(w,y,z)$ with a small correction term.

\item [Step 2] We write the event $E_1$ in an
analytic form of $(w,y,z)$ with a small correction term.

\item [Step 3] We evaluate
the integral in (\ref{intt}) using\ the results from Step 2 and the
analytic form of $h(w,y,z)$.
\end{enumerate}

\subsubsection{Step 1: $v'(x)$}

It is necessary to keep in mind that all the following derivations are on the set $\mathcal L_u'$.
Consider the change of variable that
\begin{equation}
\label{eqn:schange}
 s=  s(x):  x \to \sqrt{\Delta (u-z)}\Big( x-x_{\ast
}-\frac{y}{\Delta (u-z)}\Big) .
\end{equation}
We insert $s$ to the expansion in \eqref{expansion} and obtain that (after some elementary calculations)
\begin{eqnarray}\label{xi}
\xi (x) &=&u+w+\frac{y^{2}}{2\Delta (u-z)}-\frac{Ay^{4}}{8\Delta
^{4}(u-z)^{3}} -\frac{s^{2}}{2}-\frac{Ay^{3}}{3\Delta ^{7/2}(u-z)^{5/2}}s\\
&&-\frac{Ay^{2}}{%
4\Delta ^{3}(u-z)^{2}}s^{2}+\frac{A}{24\Delta ^{2}(u-z)}s^{4} +g(x-x_{\ast})+\vartheta (x-x_{\ast})+o(s^{4}u^{-5/4}).\notag
\end{eqnarray}%
To begin with, we are interested in
approximating%
\begin{equation}
F(x)-\frac{\int_{0}^{L}F(t)e^{\sigma \xi (t)}dt}{\int_{0}^{L}e^{\sigma \xi(t)}dt} = \frac{\int_{0}^{L}(F(x) - F(t))e^{\sigma \xi (t)}dt}{\int_{0}^{L}e^{\sigma \xi(t)}dt}.  \label{factor}
\end{equation}%
To compute the integration, it is convenient to write the terms
in the above expansion formula for $\xi(x)$  that do not  include $x$ (or equivalently $s$) as
$
c_{\ast }\triangleq \sigma \left[ u+w+\frac{y^{2}}{2\Delta (u-z)}-\frac{Ay^{4}}{%
8\Delta ^{4}(u-z)^{3}}\right] .
$
We first consider the denominator%
\begin{eqnarray*}
\int_{0}^{L}e^{\sigma \xi (x)}dx &=&e^{c_{\ast }}\int_{0}^{L}\exp \bigg \{\sigma
\big \lbrack -\frac{s^{2}}{2}-\frac{Ay^{3}}{3\Delta ^{7/2}(u-z)^{5/2}}s-\frac{%
Ay^{2}}{4\Delta ^{3}(u-z)^{2}}s^{2} \\
&&~~~~~~~~~~~~~~~~~~~~~~+\frac{A}{24\Delta ^{2}(u-z)}s^{4}+g(x-x_{\ast})+\vartheta (x-x_{\ast
}) \big ]  \bigg \}dx,
\end{eqnarray*}%
and separate it into two parts%
\begin{eqnarray}
\int_{0}^{L}e^{\sigma \xi (x)}dx &=&\int_{|x-x_{\ast}|<u^{-1/2+8\delta
}}e^{\sigma \xi (x)}dx+\int_{|x-x_{\ast}|\geq u^{-1/2+8\delta }}e^{\sigma
\xi (x)}dx.  \label{split} \\
&=&J_{1}+J_{2}.  \nonumber
\end{eqnarray}%
According to Assumption A3 and on the set $%
\{\sup_{|x|>u^{-1/2+8\delta }} [|g(x)|-\delta ^{\prime }ux^{2} ]  \leq 0\}$ ($\delta'$ can be chosen arbitrarily small), there
exists some $\varepsilon _{0}>0$ so that the minor term%
\[
J_{2}=\int_{|x-x_{\ast}|\geq u^{-1/2+8\delta }}e^{\sigma \xi (x)}dx\leq
\int_{|x-x_{\ast}|\geq u^{-1/2+8\delta }}e^{c_{\ast }-2\varepsilon
_{0}u(x-x_{\ast})^{2}}\leq e^{c_{\ast }-\varepsilon _{0}u^{16\delta }}.
\]%
We now proceed to the dominating term $J_1$. Note that, on the set $|x-x_{\ast
}|<u^{-1/2+8\delta }$, $\vartheta (x-x_{\ast})=o(u^{-1})$. Then, we obtain
that%
\begin{equation*}
\begin{split}
J_{1} &=\frac{e^{c_{\ast }+o(u^{-1})}}{\sqrt{\Delta (u-z)}}e^{\omega (u)}  \times \\
& \int_{|x-x_{\ast}|<u^{-1/2+8\delta }}
\exp\left \{\sigma \left \lbrack -\frac{%
s^{2}}{2}-\frac{Ay^{3}}{3\Delta ^{7/2}(u-z)^{5/2}}s-\frac{Ay^{2}}{4\Delta
^{3}(u-z)^{2}}s^{2}+\frac{A}{24\Delta ^{2}(u-z)}s^{4}\right]
\right \}ds,
\end{split}
\end{equation*}%
where $\omega (u)=O(\sup_{|x|\leq u^{-1/2+8\delta }}|g(x)|)$. Since $%
Var(g(x))=O(|x|^{6})$,  it is helpful to keep in mind that $\omega
(u)=O_{p}(u^{-3/2+24\delta })$.

\begin{lemma}
\label{LemInt} On the set $\mathcal L_{u}^{\prime }$, we have that%
\begin{eqnarray*}
&&\int_{|x-x_{\ast}|<u^{-1/2+8\delta }}e^{\sigma \lbrack -\frac{s^{2}}{2}-%
\frac{Ay^{3}}{3\Delta ^{7/2}(u-z)^{5/2}}s-\frac{Ay^{2}}{4\Delta ^{3}(u-z)^{2}%
}s^{2}+\frac{A}{24\Delta ^{2}(u-z)}s^{4}]}ds \\
&=&\sqrt{\frac{2\pi }{\sigma }}\exp \left\{ -\frac{Ay^{2}}{4\Delta
^{3}(u-z)^{2}}+\frac{A}{8\Delta ^{2}\sigma u}+o(u^{-1})\right\}
\end{eqnarray*}
\end{lemma}

The proof of this lemma is elementary and is provided in the supplemental material.
We insert the result of the above lemma into the expression of $J_{1}$ term,
put $J_{1}$ and $J_{2}$ terms together, and obtain that on the set $\mathcal L_{u}^{\prime }$%
\begin{equation}\label{deno}
\int_{0}^{L}e^{\sigma \xi (x)}dx=\sqrt{\frac{2\pi }{\sigma \Delta (u-z)}}%
\exp \left\{ c_{\ast }-\frac{Ay^{2}}{4\Delta ^{3}(u-z)^{2}}+\frac{A}{8\Delta
^{2}\sigma (u-z)}+\omega (u)+o(u^{-1})\right\} .
\end{equation}
We now proceed to the analysis of \eqref{factor}.
Let
$${\tau_{\ast }=x_{\ast}+\gamma _{\ast },}$$
where $\gamma _{\ast }=u^{-1/2}\Delta ^{-1/2}\sigma ^{-1/2}$. For each {$x-\tau_{\ast }=O(u^{-1/2+16\delta})$}, we define  change of variable for $x$
\begin{equation}\label{gamma}
\gamma =x-x_{\ast}-\frac{y}{\Delta (u-z)}.
\end{equation}
Note that $\xi(x)$ is approximately a quadratic function with maximum at $x_{\ast}+\frac{y}{\Delta (u-z)}$. Thus, $\gamma$ is approximately the distance to the mode of $\xi(x)$.
Similar to the derivations of Lemma \ref{LemInt} and using the results in  \eqref{deno}, the following lemma provides an approximation of \eqref{factor}. The proof is provided in the supplemental material.

\begin{lemma}\label{Lemf1}
On the set $\mathcal L_u'$, we have that
\begin{eqnarray}\label{factor1}
F(x)-\frac{\int_{0}^{L}F(t)e^{\sigma \xi (t)}dt}{\int_{0}^{L}e^{\sigma \xi
(t)}dt}&=&p(x)\gamma \exp \Big[-\frac{p^{\prime }(x)}{2p(x)\gamma }(\gamma ^{2}+\frac{1%
}{\sigma \Delta (u-z)})+\frac{p^{\prime \prime }(x)}{6p(x)}(\gamma ^{2}+%
\frac{3}{\sigma \Delta (u-z)}) \notag\\
&&~~~~~~~~~~~~~~~~~~+\frac{Ay^{3}}{3\Delta ^{4}(u-z)^{3}\gamma }+o(u^{-1})+\omega (u)\Big].
\end{eqnarray}
\end{lemma}
We apply the change of variable in \eqref{gamma} to the representation of $\xi(x)$ in \eqref{expansion}
and obtain that
\begin{eqnarray}\label{xi1}
\xi (x) &=&u+w+\frac{y^{2}}{2\Delta (u-z)}-\frac{\Delta (u-z)}{2}\gamma ^{2}
-\frac{A}{6\Delta }y(\gamma +\frac{y}{\Delta (u-z)})^{3}+\frac{Au}{24}%
(\gamma +\frac{y}{\Delta (u-z)})^{4} \notag\\
&&+g(x-x_{\ast})+\vartheta (x-x_{\ast}).
\end{eqnarray}%
We now put together \eqref{factor1} and \eqref{xi1} and obtain that for $|x-x_*|\leq u^{-1/2+8\delta}$
\begin{eqnarray}\label{dv}
v'(x)
&=&e^{\sigma u+\sigma w+\frac{\sigma y^{2}}{2\Delta (u-z)}}\times p(x)\gamma
\times e^{-\frac{\sigma \Delta u}{2}\gamma ^{2}} \\
&&\times \exp \Big\{\frac{\sigma \Delta z}{2}\gamma ^{2}-\frac{\sigma A}{6\Delta
}y(\gamma +\frac{y}{\Delta (u-z)})^{3}+\frac{\sigma Au}{24}(\gamma +\frac{y}{%
\Delta (u-z)})^{4} \notag\\
&&~~~~-\frac{p^{\prime }(x)}{2p(x)\gamma }(\gamma ^{2}+\frac{1}{\sigma \Delta
(u-z)})+\frac{p^{\prime \prime }(x)}{6p(x)}(\gamma ^{2}+\frac{3}{\sigma
\Delta (u-z)}) +\frac{Ay^{3}}{3\Delta ^{4}(u-z)^{3}\gamma }+o(u^{-1})+\omega (u)\Big\}.\notag
\end{eqnarray}%

\subsubsection{Step 2: the event $E_{1}=\{\max_{x\in \lbrack u^{-1/2+\protect\delta %
},L-u^{-1/2+\protect\delta }]}|v^{\prime }(x)|>b\}$}

By the definition of $u$ and the analytic form of \eqref{dv}, we have that
$
v'(x)\geq b=p(x_{\ast})\gamma _{\ast
}e^{\sigma u-\frac{\Delta \sigma u}{2}\gamma _{\ast }^{2}}.
$
 if and only if $\gamma >0$ and
\begin{equation}\label{11}
\begin{split}
&\sigma w+\frac{\sigma y^{2}}{2\Delta (u-z)}+\frac{\sigma \Delta z}{2}%
\gamma ^{2} -\frac{\sigma A}{6\Delta }y(\gamma +\frac{y}{\Delta (u-z)})^{3}+\frac{%
\sigma Au}{24}(\gamma +\frac{y}{\Delta (u-z)})^{4} 
\\
&-\frac{p^{\prime }(x)}{2p(x)\gamma }(\gamma ^{2}+\frac{1}{\sigma \Delta
(u-z)})+\frac{p^{\prime \prime }(x)}{6p(x)}(\gamma ^{2}+\frac{3}{\sigma
\Delta (u-z)}) \\
&+\frac{Ay^{3}}{3\Delta ^{4}(u-z)^{3}\gamma }+\log H(\gamma,u )-\log H(\gamma _{\ast},u)+\log \frac{p(x)}{p(x_{\ast})} 
\\
\geq & o(u^{-1})-\omega (u),
\end{split}
\end{equation}%
where $H$ is defined as in \eqref{HDef} and $\gamma_* = \frac{1}{\sqrt{\sigma \Delta u}}$.
We write the left-hand side of the above display as $R(\gamma ) + \log H(\gamma, u) - \log H(\gamma_*, u)$.
Note that $\partial^2_\gamma \log H(\gamma _{\ast },u)=-2\Delta \sigma u$ and the derivative of the remainder term is $\partial_\gamma R(\gamma_*) = o(1) +O(z\gamma_*)$. Thus, $\log H(\gamma, u)$ dominates the variation.
In particular, the left-hand side of \eqref{11} is maximized at
$\gamma =\gamma _{\ast }+o(u^{-1})+O(z\gamma_*/u)=u^{-1/2}\Delta ^{-1/2}\sigma ^{-1/2}+ o(u^{-1}) +O(z\gamma_*/u),$
 equivalently, at $x=x_* + \gamma_*+ {y}/{\Delta (u-z)}+o(u^{-1})+O(z\gamma_*/u).$
 Therefore,
 $\max_{|\gamma|\leq u^{-1/2+8\delta}} R(\gamma ) + \log H(\gamma, u) - \log H(\gamma_*, u) = R(\gamma_*) + o(u^{-1})+O(z^2 /u^2) .$
 This is interpreted as
\[
\max_{|x-x_*|\leq u^{-1/2+8\delta}} v^{\prime }(x)\geq b
\]%
if and only if%
\begin{equation}
\label{mmA}
\begin{split}
\mathcal{A} \triangleq & \sigma w+\frac{\sigma y^{2}}{2\Delta (u-z)}+\frac{\sigma
\Delta z}{2}\gamma _{\ast }^{2} -\frac{\sigma A}{6\Delta }y(\gamma _{\ast }+\frac{y}{\Delta (u-z)})^{3}+%
\frac{\sigma Au}{24}(\gamma _{\ast }+\frac{y}{\Delta (u-z)})^{4} 
\\
&-\frac{p^{\prime }(x)}{2p(x)\gamma _{\ast }}(\gamma _{\ast }^{2}+\frac{1}{%
\sigma \Delta (u-z)})+\frac{p^{\prime \prime }(x)}{6p(x)}(\gamma _{\ast
}^{2}+\frac{3}{\sigma \Delta (u-z)}) 
\\
&+\frac{Ay^{3}}{3\Delta ^{4}(u-z)^{3}\gamma _{\ast }}+\log \frac{p(x_{\ast} + \gamma_*+\Delta ^{-1}(u-z)^{-1}y)}{p(x_{\ast})} +O(z^2 /u^2) 
\\
\geq & o(u^{-1})-\omega (u).
\end{split}
\end{equation}%

Note that on the region $|x-x_*|> u^{-1/2+8\delta}$ we need to consider the variation of $g(x-x_*)$. On the set $\mathcal L_u'$, the variation of $v'(x)$ is dominated by $\log H(\gamma, u)$.
In particular, on the set $|x-x_*|> u^{-1/2+8\delta}$,
$$\log H(\gamma, u) - \log H(\gamma_*, u) \leq - \varepsilon_0 u(\gamma - \gamma_*)^2.$$
Furthermore, on the set $\mathcal L_u'$, we have that
$\sup_{|x|>u^{-1/2+8\delta }}(|g(x)|-\delta ^{\prime }ux^{2}) <0.$
We can choose $\delta' < \varepsilon_0/2$, then $2|g(x)| < \log H(\gamma_*, u) - \log H(\gamma, u)$ for all $|x-x_*|> u^{-1/2+8\delta}$.
Thus, on the set $\mathcal L_u'$, the maximum of $v'(x)$ is attained on $|x-x_*|\leq u^{-1/2+8\delta}$, i.e.
$$\max_{[u^{-1/2 + \delta},L- u^{-1/2+\delta}]} v'(x) > b \quad \textrm{if and only if} \quad \mathcal A> o(u^{-1}) - \omega(u).$$
The following lemma simplifies the analytic form of $\mathcal A$. The proof is provided in the supplemental material.
\begin{lemma}\label{mmmA}
The expression $\mathcal A$ can be simplified to
\begin{eqnarray*}
\mathcal{A}
&=&\sigma w+\frac{\sigma y^{2}}{2\Delta u}+\frac{\sigma }{2\Delta u^{2}}%
y^{2}z+\frac{z}{2u}+\frac{A}{24\sigma \Delta ^{2}u}+\frac{p^{\prime \prime }(x_{\ast})}{6p(x_{\ast})\sigma \Delta u} \\
&&-\frac{\sigma Ay^{4}}{8\Delta u^{3}}+\frac{y^{2}}{u^{2}}(-\frac{A}{4\Delta
^{3}}+\frac{p^{\prime \prime }(x_{\ast})}{2p(x_{\ast})\Delta ^{2}}%
)+o(u^{-1}+y^2 u^{-2})+O(z^2 /u^2).
\end{eqnarray*}
\end{lemma}
With exactly the same  development, we have
\[
\max_{x\in [u^{-1/2 + \delta}, L-u^{-1/2 + \delta} ]} [-v^{\prime }(x)]\geq b\quad
\mbox{if and only if}\quad
\mathcal{A}\geq o(u^{-1})+\omega (u).
\]
In fact, from the technical proof of Lemma \ref{mmmA}, we basically choose $\gamma = -\gamma_*+o(u^{-1})+O(z\gamma_*/u)$ and all the other derivations are the same.
We omit the repetitive details.
Thus, the event $E_1$ occurs if and  only if
$ \mathcal{A}\geq o(u^{-1})+\omega (u).$

\subsubsection{Step 3: evaluation of the integral in (\protect\ref{intt})} \label{SecInt}

\begin{lemma}
\label{LemDen} The random vector $(\xi (x),\xi ^{\prime }(x),\xi ^{\prime
\prime }(x))$ is a multivariate Gaussian random vector with mean zero and
covariance matrix
\[
\left(
\begin{array}{ccc}
1 & 0 & -\Delta \\
0 & \Delta & 0 \\
-\Delta & 0 & A%
\end{array}%
\right)
\]%
The density of $(\xi (x),\xi ^{\prime }(x),\xi ^{\prime \prime }(x))$
evaluated at $(u+w,y,-\Delta (u-z))$ is%
\[
h(w,y,z)= \frac{1}{(2\pi )^{3/2}\sqrt{\Delta(A-\Delta ^{2})}}\exp \left\{ -\frac{1}{2}%
S(w,y,z)\right\} ,
\]%
where
$
S(w,y,z)=u^{2}+w^{2}+\frac{\Delta ^{2}(w+z)^{2}}{A-\Delta ^{2}}+2u(w+\frac{%
y^{2}}{2\Delta u}).
$
\end{lemma}
The proof of the above lemma is elementary and therefore is omitted; see also Chapter 5.5 in \cite{AdlTay07}.
We insert the expression of $\mathcal A$ in Lemma \ref{mmmA} to the exponent of the density function%
\begin{eqnarray}\label{SS}
S(w,y,z) &=&u^{2}+w^{2}+\frac{\Delta ^{2}(w+z)^{2}}{A-\Delta ^{2}}+2u\left(w+%
\frac{y^{2}}{2\Delta u} \right ) \\
&=&u^{2}+w^{2}+\frac{\Delta ^{2}(w+z)^{2}}{A-\Delta ^{2}} +2u\bigg[\frac{\mathcal{A}}{\sigma }-\frac{y^{2}z}{2\Delta u^{2}}-\frac{z}{%
2\sigma u}-\frac{A}{24\sigma ^{2}\Delta ^{2}u}-\frac{p^{\prime \prime
}(x_{\ast})}{6p(x_{\ast})\sigma ^{2}\Delta u} \notag\\
&&~~~~~~~~+\frac{Ay^{4}}{8\Delta ^{4}u^{3}}-\frac{y^{2}}{u^{2}}
\left (-\frac{A}{4\sigma
\Delta ^{3}}+\frac{p^{\prime \prime }(x_{\ast})}{2p(x_{\ast})\sigma \Delta
^{2}}\right)+o(u^{-1}+y^2 u^{-2})+O(z^2 /u^2)\bigg].\notag
\end{eqnarray}%
Furhtermore, we
construct a dominating function preparing for the application of the dominated convergence theorem
\begin{eqnarray*}
S(w,y,z) &=&u^{2}+2u\mathcal{A}/\sigma +\frac{(\sqrt{A}w+\Delta
^{2}A^{-1/2}z)^{2}}{A-\Delta ^{2}}+\frac{\Delta ^{2}}{A}z^{2} \\
&&-\frac{y^{2}z}{\Delta u}-\frac{z}{\sigma }+\frac{A}{4\Delta ^{4}u^{2}}%
y^{4}-\frac{y^{2}}{u}(-\frac{A}{2\sigma \Delta ^{3}}+\frac{p^{\prime \prime
}(x _{\ast })}{p(x_{\ast })\sigma \Delta ^{2}})+ o(y^2/u) + O(z^2/u)+O(1)\\
&=&u^{2}+2u\mathcal{A}/\sigma  +\frac{(\sqrt{A}w+\Delta ^{2}A^{-1/2}z)^{2}}{A-\Delta ^{2}}+\frac{\Delta
^{2}}{A}\Big(\frac{A}{2\Delta ^{3}}\frac{y^{2}} u-z\Big)^{2}\\
&&+\frac{1}{\sigma }\Big(\frac{A}{2\Delta ^{3}}\frac{y^{2}} u-z\Big)-\frac{p^{\prime \prime }(x_{\ast })}{p(x_{\ast })\sigma \Delta ^{2}}\frac{y^{2}}u
+o(y^2/u) + O(z^2/u)+O(1) .
\end{eqnarray*}%
Note that, on the set $\mathcal L_u'$, $o(y^2/u) + O(z^2/u) = o(y^2/u +z)$ and thus,
$$S(w,y,z)\geq u^{2}+2u\mathcal{A}/\sigma +\frac{\Delta
^{2}}{A}\Big(\frac{A}{2\Delta ^{3}}\frac{y^{2}} u-z\Big)^{2}+\frac{1+o(1)}{\sigma }\Big(\frac{A}{2\Delta ^{3}}\frac{y^{2}} u-z\Big)-\frac{p^{\prime \prime }(x_{\ast })}{p(x_{\ast })\sigma \Delta ^{2}}\frac{y^{2}}u+O(1).$$
It is useful to keep in mind that $p''(x_*)<0$.
Let $\mathcal{A}_{u}=u\mathcal{A}$. Note that for each fixed $(\mathcal{A}%
_{u},y,z)$, $w\rightarrow 0$ as $u\rightarrow \infty $. Furthermore,
notice that $\omega (u)=O(\sup_{|x|\leq u^{-1/2+8\delta }}|g(x)|) = O_p(u^{-3/2+24\delta})$.
We consider change of variable from $(w,y,z)$ to $(\mathcal{A}_{u},y,z)$.
By the dominated convergence theorem and \eqref{SS}, we obtain that%
\begin{eqnarray*}
&&\Delta\int_{\mathcal L_{u}}P\left( E_1, \mathcal L_{u}^{\prime }|w,y,z\right) h(w,y,z)dwdydz \\
&=&\frac{\sqrt\Delta }{(2\pi )^{3/2}\sqrt{A-\Delta ^{2}}} \times \int_{\mathcal L_{u}}P\left(\mathcal{A}>\omega (u), \mathcal L_{u}^{\prime } \vert \,  w, y,z\right) e^{-\frac{1}{2}S(w,y,z)}dwdydz \\
&\sim&\frac{\sqrt\Delta }{(2\pi )^{3/2}\sqrt{A-\Delta ^{2}}} \times \int_{\mathcal L_{u}} I  \left(\mathcal{A}_u>0\right) e^{-\frac{1}{2}S(w,y,z)}\frac{d\mathcal A_u}{\sigma u}dydz
\end{eqnarray*}
For the last step, we use the fact that $P(\mathcal L_u'|w,y,z)\rightarrow 1$ and $P( \mathcal A > \omega(u), \mathcal L_{u}^{\prime }|w,y,z)\to I(\mathcal{A}_u>0)$ as $u\to \infty$.
We insert the expression $S(w,y,z)$ as in \eqref{SS} and set $w=0$ (by the dominated convergence theorem and the fact that for fixed $\mathcal A_u$, $y$, and $z$, we have $w\to 0$ as $u\to \infty$). The above the display is
\begin{eqnarray*}
&\sim&\frac{\sqrt\Delta }{(2\pi )^{3/2}\sqrt{A-\Delta ^{2}}}u^{-1}e^{-u^{2}/2+%
\frac{A}{24\sigma ^{2}\Delta ^{2}}+\frac{p^{\prime \prime }(x_{\ast})}{%
6p(x_{\ast})\sigma ^{2}\Delta }} \times \int_{0}^\infty \frac 1 \sigma e^{-\mathcal A_u/\sigma}d\mathcal A_u\\
&&\times \int\exp 
\left(
-\frac{1}{2}\left[
\frac{\Delta
^{2}z^{2}}{A-\Delta ^{2}}-\frac{z}{\sigma }-\frac{y^{2}z}{\Delta u} +\frac{A}{4\Delta ^{4}}\frac{y^{4}}{u^{2}}-\frac{y^{2}}
{u} \left(-\frac{A}{%
2\sigma \Delta ^{3}}+\frac{p^{\prime \prime }(x _{\ast })}{p(x_{\ast })\sigma \Delta ^{2}}
\right)
\right]
\right ) dydz.
\end{eqnarray*}%
We use the change of variable that $y_{u}=u^{-1/2}y$%
\begin{eqnarray}
&\sim&\frac{\sqrt\Delta }{(2\pi )^{3/2}\sqrt{A-\Delta ^{2}}}e^{\frac{A}{%
24\sigma ^{2}\Delta ^{2}}+\frac{p^{\prime \prime }(x_{\ast})}{6p(x_{\ast
})\sigma ^{2}\Delta }}u^{-1/2}e^{-u^{2}/2}  \label{int} \\
&&\times \int\exp \left ( 
-\frac{1}{2}
\left[
\frac{\Delta ^{2}z^{2}}{A-\Delta ^{2}}-\frac{z}{\sigma }-\frac{%
y_u^{2}z}{\Delta } + \frac{A}{4\Delta ^{4}}y_u^{4}-y_u^{2}
\left(-\frac{A}{2\sigma \Delta ^{3}}+\frac{p^{\prime \prime }(x _{\ast })}{p(x _{\ast })\sigma \Delta ^{2}}%
\right)
\right]
\right) dy_udz  \nonumber\\
&=& D\times u^{-1/2} e^{-u^2/2}.\notag
\end{eqnarray}%
This corresponds to the first term of the approximation in the statement of the theorem.
%
%
%
%
%
%

\subsection{The approximation of $P(E_{3})\label{SecE3}$}

The analysis of $P(E_{2})$ and $P(E_{3})$ are completely analogous.
Therefore, we only provide the derivation for $P(E_{3})$. The difference between the analyses of $P(E_{3})$ and $P(E_{1})$ is that the integrals in the
factor (\ref{factor}) are truncated by the boundary and therefore most of the
calculations are related to conditional Gaussian distributions. We redefine some notation. Let $u_L$ and $\zeta_L$ be defined as in Section \ref{SecThm} prior to the statement of the theorem.
We first define
$t_{L}=L-\frac{\zeta _{L}}{\sqrt{\Delta \sigma u_{L}}}$
that is the location where $\xi(x)$ is likely  to have a high excursion given that $v'(x)$ has a high excursion at the right boundary $L$. We will perform Taylor expansion by conditioning on the field at $t_L$. We  redefine the notation $(w,y,z)$ as
$\xi (t_{L})=u_L+w$, $\xi ^{\prime }(t_{L})=y$, and $\xi ^{\prime \prime
}(t_{L})=-\Delta (u_{L}-z).$
Furthermore, we consider the following change of variables ``$\gamma$'' and ``$s$''
\begin{equation}\label{change}
x=\gamma +t_{L}+\frac{y}{\Delta (u_{L}-z)},\quad t=t_{L}+\frac{y}{\Delta (u_{L}-z)}+\frac{s}{\sqrt{\Delta (u_{L}-z)}}.
\end{equation}
With simple calculations, we have that
\begin{equation}\label{boundary}
t\leq L\Longleftrightarrow s\leq \sqrt{\frac{(1-z/u_L)}{\sigma }}\zeta _{L}-%
\frac{y}{\sqrt{\Delta (u_L-z)}}.
\end{equation}
Furthermore, it is useful to keep in mind that $v'(x)$ is maximized when $\gamma$ is of order $u_L^{-1/2}$.
Let $g(x)$ be the remainder process such that
$\xi(x) = E(\xi(x) | w,y,z) + g(x - t_L).$
Similar to the analysis of $P(E_1)$, we first localize the event via the following proposition.

\begin{proposition}
\label{PropLocal2} Using the notations in Theorem \ref{ThmMain}, under conditions A1 - A3, consider%
\begin{eqnarray*}
\mathcal C_{u_L}&=&\{|w|>u_{L}^{3\delta }\}\bigcup \{|y|>u_{L}^{1/2+4\delta}\}\bigcup \{|z|>u_{L}^{1/2+4\delta }\}\\
&&~~\bigcup\Big\{\sup_{|x|>u_L^{-1/2+8\delta }}[|g(x)|-\delta^{\prime }u_Lx^{2}]>0\Big\}
\bigcup \Big \{\sup_{|x|\leq u_L^{-1/2+8\delta }}|g(x)|>u_L^{-1/2+\delta ^{\prime}} \Big\}
\end{eqnarray*}
Then, for any $\delta >0$ and $\delta ^{\prime }>24\delta $, we have that
$P(\mathcal C_{u_L};E_{3})=o(u_{L}^{-1}e^{-u_{L}^{2}/2}).$
\end{proposition}
Let $\mathcal L_{u_L}^*= \mathcal C_{u_L}^c$ and we only need to consider $P(\mathcal L_{u_L}^*,E_{3})$.
With a similar derivation as that for $P(E_{1})$, the following lemma provides an estimate of
${\int_{0}^{L}(F(x)-F(t))e^{\sigma \xi(t)}dt}/{\int_{0}^{L}e^{\sigma \xi(t)}dt}.$
The proof is provided in the supplemental material.

\begin{lemma}\label{Lemf2}
On the set $\mathcal L_{u_L}^*$, we have that
\begin{equation}
\label{f2}
\begin{split}
&\frac{\int_{0}^{L}(F(x)-F(t))e^{\sigma \xi(t)}dt}{\int_{0}^{L}e^{\sigma \xi(t)}dt} =
\frac{1}{\sqrt{\Delta \sigma u_{L}}} \times \exp \left( \frac{z}{2u_{L}}-\frac{A}{24\Delta ^{2}\sigma u_{L}}E\left( Z^{4}|Z\leq \zeta
_{L}\right) +\lambda (u_{L})+\omega (u_{L})\right) \times
\\
& \left\{E\left[p(x)(\gamma \sqrt{\sigma \Delta (u_{L}-z)}-Z)
-\frac{p'(x)}{2\sqrt{\sigma \Delta u_{L}}}(\gamma \sqrt{\sigma \Delta (u_{L}-z)}-Z)^{2} ~\left \vert~ Z\leq \sqrt{1-\frac{z}{u_{L}}}\zeta _{L}-\sqrt{\frac{\sigma }{\Delta (u_{L}-z)}}y \right. \right]\right. 
\\
&~~+E\left. \left[\frac{p^{\prime
\prime }(x)}{6\sigma \Delta u_{L}}(\gamma \sqrt{\sigma \Delta u_{L}}-Z)^{3}
    +\frac{Ap(x)}{24\Delta ^{2}\sigma ^{2}u_{L}}Z^{4}(\gamma \sqrt{\sigma \Delta u_{L}}-Z)~\left \vert ~Z\leq \sqrt{1-\frac{z}{u_{L}}}\zeta _{L}-\sqrt{%
\frac{\sigma }{\Delta (u_{L}-z)}}y \right. \right]\right\}
\end{split}
\end{equation}%
where $\lambda (u_{L})=O(y^{3}/u_{L}^{5/2}+y^{2}/u_{L}^{2} + y /u^{3/2})+o(u_{L}^{-1}+u_{L}^{-1}z),$ $\omega (u)=O(\sup_{|x|\leq u^{-1/2+8\delta }}|g(x)|),$
and  $Z$ is a standard Gaussian random variable.
\end{lemma}

Inside the ``$\left\{ \ \right\}$" of the above approximation,
the first expectation term is the dominating term and the second term is of order $%
O(u^{-1})$. The next lemma presents an approximation of $v'(x)$.

\begin{lemma}\label{Lemvprime}
On the set $\mathcal L_{u_L}^*$, we have that
\begin{equation}
\label{vpp}
\begin{split}
v'(x) =&\exp \left( \lambda (u_{L})+ o( yu_L^{-1}) +O(y^2zu_L^{-2})+\omega (u_{L})+\sigma u_{L}+\sigma w+\frac{\sigma
y^{2}}{2\Delta u_{L}}+\frac{A\sigma
u_{L}}{24}\gamma ^{4}\right)
\\
&\times\frac{1}{\sqrt{\Delta \sigma u_{L}}}\exp \left( \frac{z}{2u_{L}}-\frac{A}{24\Delta
^{2}\sigma u_{L}}E( Z^{4}|Z\leq \zeta _{L}) \right)
\\
&\times H_{L,x}\left(\gamma \sqrt{\sigma \Delta (u_{L}-z)},\sqrt{1-\frac{z}{u_{L}}}\zeta _{L}-\sqrt{\frac{\sigma }{\Delta (u_{L}-z)}}y; u_L \right)
 \\
&\times\exp \left\{ \frac{E\Big[\frac{p^{\prime \prime }(x)}{6\sigma \Delta u_{L}}(\gamma \sqrt{\sigma\Delta u_{L}}-Z)^{3}+\frac{Ap(x)}{24\Delta ^{2}\sigma ^{2}u_{L}}Z^{4}(\gamma
\sqrt{\sigma \Delta u_{L}}-Z)\left\vert Z\leq \zeta _{L}\right. \Big ]}{%
p(x)E(\gamma \sqrt{\sigma\Delta u_{L}}-Z\vert Z\leq \zeta _{L} )}\right\} ,
\end{split}
\end{equation}%
where
$
H_{L,y}(x,\zeta; u)\triangleq e^{-\frac{x^{2}}{2}} \times
E\Big[p(y)(x-Z)- \frac{p^{\prime}(y)}{2\sqrt{\Delta \sigma u}}(x-Z)^{2}   ~\Big | ~
Z\leq \zeta \Big].
$
\end{lemma}

Note that the definition of $H_{L,y}(x,\zeta;u)$ is slightly different from $H_L(x,\zeta,u)$ defined as in \eqref{hl}. In particular, if we let $y=L$, then $H_{L,y}(x,\zeta;u) = H_L(x,\zeta;u)$.
Furthermore, according to the change of variable in \eqref{change}, $x\leq L$ if and only if
\begin{equation}\label{ct}
\gamma \sqrt{\sigma \Delta (u_{L}-z)}\leq \sqrt{1-\frac{z}{u_{L}}}\zeta _{L}-\sqrt{\frac{\sigma }{\Delta (u_{L}-z)}}y.
\end{equation}
Thus,  the maximization of $v'(x)$ (in choosing the variable $\gamma$) is subject to  the above constraint.
According the definition of $u_L$ in \eqref{ul} and the notation $G_{{L}}(\zeta; u_L )=\sup_{x\leq \zeta }\log |H_{L}(x,\zeta,u_L )|$,
we have that $\max_{x\in [L - u^{-1/2 +\delta},L]}|v'(x)|>b$ if and only if%
\begin{eqnarray}\label{22}
&&\max_{x\in \lbrack L-u_{L}^{-1/2+\delta },L]}\lambda
(u_{L})+\omega (u_{L})+ o( yu_L^{-1}) +O(y^2zu_L^{-2}) \\
&&+\sigma w+\frac{\sigma y^{2}}{2\Delta u_{L}}+\frac{A\sigma u_{L}}{24}\gamma ^{4}+\frac{z}{2u_{L}}-\frac{AE\left( Z^{4}|Z\leq \zeta _{L}\right) }{24\Delta ^{2}\sigma u_{L}} \notag\\
&&+\log\left \vert 
 H_{L,x}\left(\gamma \sqrt{\sigma \Delta (u_{L}-z)},\sqrt{1-\frac{z}{u_{L}}}\zeta _{L}-\sqrt{\frac{\sigma }{\Delta (u_{L}-z)}}y; \, u_L\right)\right \vert - G_L(\zeta_L; u_L)
\notag\\
&&+\frac{E[\frac{%
p^{\prime \prime }}{6\sigma \Delta u_{L}}(\gamma \sqrt{\sigma \Delta u_{L}}%
-Z)^{3}+\frac{Ap}{24\Delta ^{2}\sigma ^{2}u_{L}}Z^{4}(\gamma \sqrt{\sigma
\Delta u_{L}}-Z)~~\vert ~~Z\leq \zeta _{L} ]}{p(x)E(\gamma \sqrt{\sigma\Delta u_{L}}-Z\vert Z\leq \zeta _{L} )} ~~>0.\notag
\end{eqnarray}%
We now proceed to the evaluation of $P(E_3)$ that consists of two cases.

We first consider the case that  $\vert \sqrt{1-\frac{z}{u_{L}}}\zeta _{L}-\sqrt{\frac{\sigma }{\Delta (u_{L}-z)}}y - \zeta_L\vert\leq \varepsilon $.
Note that the major variation of the left-hand-side of \eqref{22} is dominated by
\begin{equation}\label{HH}
\log\left \vert  H_{L,x}\left(\gamma \sqrt{\sigma \Delta (u_{L}-z)},\sqrt{1-\frac{z}{u_{L}}}\zeta _{L}-\sqrt{\frac{\sigma }{\Delta (u_{L}-z)}}y;u_L\right)\right \vert.
\end{equation}
Thanks to the discussion in Remark \ref{RemH}, the above expression is maximized at (subject to the constraint \eqref{ct})
$\gamma \sqrt{\sigma \Delta (u_{L}-z)}= \sqrt{1-\frac{z}{u_{L}}}\zeta _{L}-\sqrt{\frac{\sigma }{\Delta (u_{L}-z)}}y,$
that is,
\begin{equation}\label{gammam}
\gamma =\frac{\zeta _{L}}{\sqrt{\Delta \sigma u_{L}}}-\frac{y}{\Delta
(u_{L}-z)}.
\end{equation}
Recall the change of variable in \eqref{change}, this corresponds to $x=L$. That is, the maximum is attained on the boundary $x=L$.
Then, we can replace $H_{L,x}$ in \eqref{22} by $H_{L,L} = H_L$.
Let $\gamma _{L}=\frac{\zeta _{L}}{\sqrt{\sigma \Delta u_{L}}}$. For the particular choice of $\gamma$ in \eqref{gammam}, we have that $\gamma^4 = \gamma_L^4 + o(y^2/u_L^2)$.
We have that $\max_{x\in \lbrack L-u_{L}^{-1/2+\delta },L]}|v_{L}^{\prime }(x)|>b$ if and only if $\mathcal A \geq \omega (u_{L})$ where
\begin{equation}
\label{AA}
\begin{split}
\mathcal{A} \triangleq&\lambda (u_{L})+ o( yu_L^{-1}) +O(y^2zu_L^{-2})+\sigma w+\frac{\sigma y^{2}}{2\Delta u_{L}}+
\frac{A\sigma u_{L}}{24}\gamma
_{L}^{4}+\frac{z}{2u_{L}}-\frac{AE\left( Z^{4}|Z\leq \zeta _{L}\right) }{%
24\Delta ^{2}\sigma u_{L}} 
\\
&+ G_{L}\Big (\sqrt{1-\frac{z}{u_{L}}}\zeta _{L}-\sqrt{\frac{\sigma }{\Delta (u_{L}-z)}}y; u_L \Big) - G_{L}(\zeta_L; u_L) 
\\
&+\frac{E[\frac{p^{\prime \prime }(L)}{6\sigma \Delta u_{L}}(\gamma _{L}\sqrt{\sigma
\Delta u_{L}}-Z)^{3}+\frac{Ap(L)}{24\Delta ^{2}\sigma ^{2}u_{L}}Z^{4}(\gamma
_{L}\sqrt{\sigma \Delta u_{L}}-Z)\left\vert Z\leq \zeta _{L}\right. ]}{p(L)E(\zeta_L-Z\vert Z\leq \zeta _{L} )} .
\end{split}
\end{equation}%

\begin{lemma}\label{LemA}
The expression $\mathcal A$ can be simplified to
\begin{eqnarray*}
\mathcal{A} =\lambda (u_{L})+ o( yu_L^{-1}) +O(y^2zu_L^{-2})+\sigma w+\frac{\sigma y^{2}}{2\Delta u_{L}}+\frac{z}{2u_{L}}+\frac{\kappa_L }{u_{L}}
-\frac{\Xi _{L}+o(1)}{2}\Big(\frac{\zeta _{L}z}{2u_{L}}+\sqrt{\frac{%
\sigma }{\Delta (u_{L}-z)}}y\Big)^{2},
\end{eqnarray*}%
where $\kappa_L$ is given as in \eqref{kappa}.
\end{lemma}
With the above lemma, we rewrite $S(w,y,z)$ as
\begin{eqnarray*}
S(w,y,z) &=&u_{L}^{2}+w^{2}+\frac{\Delta ^{2}(w+z)^{2}}{A-\Delta ^{2}} +o(1) + o(y^2)\\
&&+2u_{L}\Big[\mathcal{A}/\sigma-\frac{z}{2\sigma u_{L}} -\frac{\kappa_L }{\sigma u_{L}}
+\frac{\Xi _{L}+o(1)}{2\sigma }\Big(\frac{\zeta _{L}z}{2u_{L}}+\sqrt{\frac{%
\sigma }{\Delta (u_{L}-z)}}y\Big)^{2}\\
&&~~~~~~~~~~~+\lambda (u_{L})+ o( yu_L^{-1}) +O(y^2zu_L^{-2})\Big].
\end{eqnarray*}%
Similar to the derivation of \eqref{int}, by the dominated convergence theorem, we have that
\begin{equation}\label{main}
\begin{split}
&P\left( \max_{x\in \lbrack L-u_{L}^{-1/2+\delta },L]}|v^{\prime
}(x)|>b~;~ \mathcal  L_{u_L}^*;~~ \left \vert \sqrt{1-\frac{z}{u_{L}}}\zeta _{L}-\sqrt{\frac{\sigma }{\Delta (u_{L}-z)}}y - \zeta_L \right \vert\leq \varepsilon \right)\\
 \sim &\frac{\sqrt\Delta }{(2\pi )^{3/2}\sqrt{A-\Delta ^{2}}}%
u_{L}^{-1}e^{-u_{L}^{2}/2+\frac{\kappa _{L}}{\sigma }}
\times \int \exp \left( -\frac{1}{2}\left(\frac{\Delta ^{2}z^{2}}{A-\Delta ^{2}}
-\frac z \sigma
+\frac{\Xi _{L}}{\Delta }y^{2}\right) \right)dydz\\
=& D_L \times u_L^{-1} \times e^{-u_L^2/2}.
\end{split}
\end{equation}%
The following lemma presents the case that $\left \vert \sqrt{1-\frac{z}{u_{L}}}\zeta _{L}-\sqrt{\frac{\sigma }{\Delta (u_{L}-z)}}y - \zeta_L \right \vert\geq \varepsilon $.

\begin{lemma}\label{LemMinor}
Under the conditions in Theorem \ref{ThmMain}, we have that
\begin{eqnarray*}
P\left( \max_{x\in \lbrack L-u_{L}^{-1/2+\delta },L]}|v^{\prime
}(x)|>b; \mathcal L_{u_L}^*; ~~\left\vert \sqrt{1-\frac{z}{u_{L}}}\zeta _{L}-\sqrt{\frac{\sigma }{\Delta (u_{L}-z)}}y - \zeta_L\right\vert\geq \varepsilon \right) = o(1)u_{L}^{-1}e^{-u_{L}^{2}/2}.
\end{eqnarray*}
\end{lemma}
Combining \eqref{main}, Lemma \ref{LemMinor}, and the localization result in Proposition \ref{PropLocal2}, we have that
$$P\Big( \max_{x\in [ L-u_{L}^{-1/2+\delta },L]}|v^{\prime
}(x)|>b\Big)\sim D_{L}\times u_{L}^{-1}e^{-u_{L}^{2}/2}.
 $$

\paragraph{Approximation of $P(E_2).$}
The analysis of $P(E_2)$ is completely analogous. In particular, we let
$t_0 = \frac{\zeta_0}{\sqrt{\Delta\sigma u_0}}$, $\xi(t_0) = u_0 + w$, $\xi'(t_0) = y$, and $\xi''(t_0)= -\Delta(u-z)$
and further adopt change of variables
$x= t_0 +\frac{y}{\Delta (u_0-z)} - \gamma$ and $t= t_0 + \frac{y}{\Delta(u_0-z) }-\frac{s}{\sqrt{\Delta(u_0 -z)}}.$
Then the calculations are exactly the same as those of $P(E_3)$. Therefore, we omit the repetitive derivations and provide the result  that
$
P( \max_{x\in \lbrack 0,u_{L}^{-1/2+\delta }]}|v^{\prime }(x)|>b)
\sim  D_0 \times u_0^{-1}\times e^{-u_0 ^2 /2}.
$
With the inequality \eqref{bern} and \eqref{minorint}, we conclude the proof.

\bibliographystyle{plain}
\bibliography{bibstat,bibprob,RefGrant,PM}

\newpage
\appendix

\centerline{\bf \LARGE Supplemental Material}

\section{Proof of Theorem \protect\ref{ThmHomo}} \label{SecH}

Similar to the proof of Theorem \ref{ThmMain}, we consider the event $E_{1}$%
, $E_{2}$, and $E_{3}$ separately.
By homogeneity and symmetry, $P(E_{2})=P(E_{3})$. The approximations of $P(E_2)$ and  $P(E_{3}) $ are identical to those obtained in Section \ref{SecE3} by setting $p(x)\equiv p_0$. Therefore,%
\[
P(E_{2})=P(E_{3})\sim D_{h}u_{h}^{-1}e^{-u_{h}^{-2}/2}.
\]%
From the derivation of $P(E_{2})$ in the previous proof, we obtain that $P(E_{2}\cap E_{3})=o(P(E_{2}))$. For the rest of the proof, we show that $P(E_{1})=o(P(E_{2}))$ and thus $P(E_{1}\cap E_{2})=o(P(E_{2}))$.

\paragraph{Approximation of $P(E_{1})$.}
Let $H(x,u)$ be as defined for Theorem \ref{ThmMain} and $u$ solve%
\[
p_{0} H(\gamma _{\ast }(u),u) e^{\sigma u }=b,
\]%
where $\gamma _{\ast }(u)=u^{-1/2}\Delta ^{-1/2}\sigma ^{-1/2}$. For the rest of the proof, we will show that
\begin{equation}\label{e1}
P(E_1) = O(1)e^{-\frac{u^{2}}{2}+O(u^{\varepsilon})}.
\end{equation}
for any $\varepsilon>0$.
According to the discussion in Section \ref{SecHeu}, there exists an $\varepsilon_0 >0$ such that $u > u_h + \varepsilon_0$ and thus $e^{-\frac{u^{2}}{2}+O(u^{\varepsilon})} = o(1)u_{h}^{-1}e^{-u_{h}^{-2}/2}$.
If the above bound in \eqref{e1} can be established, then we can conclude the proof.


First, we  derive an approximation for
\[
\alpha (u,\varepsilon )=P\Big(\max_{x\in \lbrack \frac L 2 -u^{-1/2+\varepsilon
},\frac L 2 + u^{-1/2+\varepsilon }]}|v^{\prime }(x)|>b\Big) ,
\]%
where $\varepsilon>0$ is chosen small enough.
Then, we split the region $[0,L]$ into $N=\frac{L}{2u^{-1/2+\varepsilon }}$ many
intervals each of which is a location shift of $[0,2u^{-1/2+\varepsilon }]$, i.e. $[2ku^{-1/2+\varepsilon
},2ku^{-1/2+\varepsilon }+2u^{-1/2+\varepsilon }]$.
Thanks to the homogeneity of $\xi(x)$, the approximations for
\[
P\left(\max_{x\in \lbrack 2ku^{-1/2+\varepsilon },2ku^{-1/2+\varepsilon
}+2u^{-1/2+\varepsilon }]}|v^{\prime }(x)|>b\right)
\]%
are the same for all $1\leq k \leq N-2$. Then,  we have%
\begin{eqnarray*}
&&P\left( \cup _{k=1}^{N-2}\{\max_{x\in \lbrack 2ku^{-1/2+\varepsilon
},2ku^{-1/2+\varepsilon }+2u^{-1/2+\varepsilon }]}|v^{\prime }(x)|>b\}\right)
\leq (1+o(1))\frac{L}{2u^{-1/2+\varepsilon }}\alpha (u,\varepsilon ).
\end{eqnarray*}%

In what follows, we derive an approximation for $\alpha (u,\varepsilon )$. The derivation is similar to the proof of the Theorem \ref{ThmMain}.
Therefore, we omit the details and only lay out the key steps and the major differences.
We expand $\xi
(x)$ around $x= \frac L 2$ conditional on (by redefining the notations)
$$\xi (\frac L 2)=u+w, \quad \xi ^{\prime }(\frac L 2)=y, \quad \xi ^{\prime \prime }(\frac L 2)=-\Delta (u-z)$$
and obtain that
\begin{eqnarray}
\xi (x) &=&u+w+\frac{y^{2}}{2\Delta (u-z)}-\frac{\Delta (u-z)}{2}\left( x-%
\frac{y}{\Delta (u-z)}\right) ^{2}  \nonumber \\
&&~~~-\frac{Ay}{6\Delta }x^{3}+\frac{Au}{24}x^{4}+g(x-\frac L 2)+\zeta (x-\frac L 2).  \nonumber
\end{eqnarray}%
Similarly, we have the following proposition for localization.

\begin{proposition}
\label{PropLocal1} For $\delta' > 3 \varepsilon$, let
\begin{eqnarray*}
\mathcal G_{u} &=&\{|w|>u^{3\varepsilon }\}\cup \{|y|>u^{1/2+4\varepsilon }\}\cup
\{|z|>u^{1/2+4\varepsilon }\} \\
&&\cup \Big\{ \sup_{x\notin \lbrack -u^{-1/2+\varepsilon },u^{-1/2+\varepsilon
}]}|g(x)|-\delta ^{\prime }ux^{2}>0\Big\} \cup \Big\{ \sup_{x\in \lbrack
-u^{-1/2+\varepsilon },u^{-1/2+\varepsilon }]}|g(x)|>u^{-1/2+\delta ^{\prime}}\Big\}
\end{eqnarray*}%
Under the conditions of Theorem \ref{ThmHomo}, we have%
\[
P(\mathcal G_{u};\max_{x\in \lbrack \frac L 2 -u^{-1/2+\varepsilon },\frac L 2+u^{-1/2+\varepsilon
}]}|v^{\prime }(x)|>b)=o(1)e^{-u^{2}/2}.
\]
\end{proposition}

Let%
\[
\mathcal L_{u}=\mathcal G_{u}^{c}.
\]%
We now proceed to the factor
\[
F(x)-\frac{\int_{0}^{L}F(t)e^{\sigma \xi (t)}dt}{\int_{0}^{L}e^{\sigma \xi
(t)}dt}.
\]%
Following exactly the same derivation as Lemma \ref{Lemf1} in Section \ref{sec:3.1} and noting that $p(x)\equiv p_{0}$, we have that%
\[
F(x)-\frac{\int_{0}^{L}F(t)e^{\sigma \xi (t)}dt}{\int_{0}^{L}e^{\sigma \xi
(t)}dt}=p_{0}\gamma \exp \left\{ \frac{Ay^{3}}{3\Delta ^{4}(u-z)^{3}\gamma }%
+o(u^{-1})+\omega (u)\right\} ,
\]%
where we redefine a change of variable similar  to \eqref{gamma} as
\[
\gamma =x - \frac L 2 -\frac{y}{\Delta (u-z)}.
\]%
Thus, similar to \eqref{dv}, we obtain that
\begin{eqnarray*}
v^{\prime }(x) &=&e^{\sigma \xi (x)}\left[ F(t)-\frac{\int_{0}^{L}F(t)e^{%
\sigma \xi (t)}dt}{\int_{0}^{L}e^{\sigma \xi (t)}dt}\right] \\
&=&e^{\sigma u+\sigma w+\frac{\sigma y^{2}}{2\Delta (u-z)}}\times
p_{0}\gamma e^{-\frac{\sigma \Delta u}{2}\gamma ^{2}} \\
&&\times \exp \Big\{\frac{\sigma \Delta z}{2}\gamma ^{2}-\frac{\sigma A}{6\Delta
}y(\gamma +\frac{y}{\Delta (u-z)})^{3}+\frac{\sigma Au}{24}(\gamma +\frac{y}{%
\Delta (u-z)})^{4} \\
&&~~~~~~~~~~~~+\frac{Ay^{3}}{3\Delta ^{4}(u-z)^{3}\gamma }+o(u^{-1})+\omega (u)\Big\}.
\end{eqnarray*}%
We further simplify the above display and obain that%
\begin{eqnarray*}
v^{\prime }(x) &=&e^{\sigma u+\sigma w+\frac{\sigma y^{2}}{2\Delta (u-z)}%
}\times p_{0}\gamma e^{-\frac{\sigma \Delta u}{2}\gamma ^{2}} \\
&&\times \exp\Big \{\frac{\sigma \Delta z}{2}\gamma ^{2}-\frac{\sigma A\gamma
^{2}}{4\Delta ^{2}u}y^{2}-\frac{\sigma A}{8\Delta ^{4}u^{3}}y^{4}+y^{3}\Big[
\frac{\sigma A\gamma }{3\Delta ^{3}u^{2}}-\frac{A}{3\Delta ^{4}u^{3}\gamma }%
\Big] \\
&&~~~~~~~~~+O(u^{-1})+\omega (u)\Big\}.
\end{eqnarray*}%
For all {$|y|\leq (1+\varepsilon^{\prime })\Delta u^{1/2+\varepsilon }$}, we have that
\begin{eqnarray}\label{mmm}
\max_{x\in \lbrack \frac L 2 -u^{-1/2+\varepsilon },\frac L 2+ u^{-1/2+\varepsilon }]} v^{\prime }(x)
&\leq &\max_{x\in \lbrack \frac L 2 - (1+ 2\varepsilon')u^{-1/2+\varepsilon },\frac L 2+(1+ 2\varepsilon') u^{-1/2+\varepsilon }]} v^{\prime }(x)\notag\\
&=&e^{\sigma u+\sigma w+\frac{\sigma y^{2}}{2\Delta (u-z)%
}}\times p_{0}\gamma _{\ast }e^{-\frac{\sigma \Delta u}{2}\gamma _{\ast
}^{2}} \notag\\
&&\times \exp \Big\{\frac{\sigma \Delta z}{2}\gamma _{\ast }^{2}-\frac{\sigma
A\gamma _{\ast }^{2}}{4\Delta ^{2}u}y^{2}-\frac{\sigma A}{8\Delta ^{4}u^{3}}%
y^{4}+y^{3}\Big[ \frac{\sigma A\gamma _{\ast }}{3\Delta ^{3}u^{2}}-\frac{A}{%
3\Delta ^{4}u^{3}\gamma _{\ast }}\Big] \notag\\
&&~~~~~~~~~+O(u^{-1}+ z^2 u^{-2})+\omega (u)\Big\}.
\end{eqnarray}%
That is, $v^{\prime }(x)$ is maximized when $x=\frac L 2 +\gamma_{\ast }+\frac{y}{\Delta (u-z)}+o(u^{-1}) +O(z\gamma_*/u)$. Since $\gamma _{\ast }=\Delta
^{-1/2}\sigma ^{-1/2}u^{-1/2}$, then%
\[
\frac{\sigma A\gamma _{\ast }}{3\Delta ^{3}u^{2}}-\frac{A}{3\Delta
^{4}u^{3}\gamma _{\ast }}=0.
\]%
Thus, we have that
\begin{eqnarray*}
\max_{x\in \lbrack \frac L 2 -u^{-1/2+\varepsilon },\frac L 2 +u^{-1/2+\varepsilon }]}v^{\prime
}(x)
&>&b
\end{eqnarray*}%
implies that 
\begin{eqnarray*}
\mathcal{A} &\triangleq &\sigma w+\frac{\sigma y^{2}}{2\Delta (u-z)}+\frac{z}{2u}-%
\frac{A}{4\Delta ^{3}u^{2}}y^{2}-\frac{\sigma A}{8\Delta ^{4}u^{3}}y^{4}+O( z^2 u^{-2}) +O(u^{-1})\\
&\geq &\omega (u).
\end{eqnarray*}%
Corresponding to the analysis in Section \ref{SecInt}, the next step is to insert $\mathcal{A}$ to $S(w,y,z)$ and obtain that%
\begin{eqnarray*}
S(w,y,z) &=&u^{2}+w^{2}+\frac{\Delta ^{2}(w+z)^{2}}{A-\Delta ^{2}}+2u(w+%
\frac{y^{2}}{2\Delta u}) \\
&=&u^{2}+\frac{(\sqrt{A}w+\Delta ^{2}A^{-1/2}z)^{2}}{A-\Delta ^{2}}+\frac{%
\Delta ^{2}}{A}z^{2} \\
&&~~~+2u\frac{\mathcal{A}}{\sigma }-\frac{y^{2}z}{\Delta u}-\frac{z}{\sigma }+%
\frac{A}{2\Delta ^{3}\sigma }\frac{y^{2}}{u}+\frac{A}{4\Delta ^{4}}\frac{%
y^{4}}{u^{2}} + O(z^2/u)+O(1)\\
&=&u^{2}+\frac{(\sqrt{A}w+\Delta ^{2}A^{-1/2}z)^{2}}{A-\Delta ^{2}}+\frac{2u%
\mathcal{A}}{\sigma } \\
&&~~~+\frac{\Delta ^{2}}{A}z^{2}-z\Big(\frac{y^{2}}{\Delta u}+\frac{1}{\sigma }\Big)+%
\frac{A}{4\Delta ^{2}}\Big(\frac{y^{2}}{\Delta u}+\frac{1}{\sigma }\Big)^{2}-\frac{A}{4\Delta ^{2}\sigma ^{2}} + O(z^2/u)+O(1)\\
&=&u^{2}+\frac{(\sqrt{A}w+\Delta ^{2}A^{-1/2}z)^{2}}{A-\Delta ^{2}}+\frac{2u%
\mathcal{A}}{\sigma } \\
&&~~~+\Big[\frac{\Delta z}{\sqrt{A}}-\frac{\sqrt{A}}{2\Delta }\Big(\frac{y^{2}}{\Delta u}+\frac{1}{\sigma }\Big)\Big]^{2}-\frac{A}{4\Delta ^{2}\sigma ^{2}}+ O(u^{8\varepsilon}).
\end{eqnarray*}%
For the last step in the above derivation, we use the fact that, on the set $\mathcal L_u$, $O(z^2/u) = O(u^{8\varepsilon})$.
Thus,
\begin{eqnarray*}
&&P\left( \max_{x\in \lbrack -u^{-1/2+\varepsilon },u^{-1/2+\varepsilon
}]}|v^{\prime }(x)|>b\right) \\
&=&\Delta\int_{\mathcal L_{u}}h(w,y,z)P(\max_{x\in \lbrack -u^{-1/2+\varepsilon
},u^{-1/2+\varepsilon }]}|v^{\prime }(x)|>b|w,y,z)dwdydz \\
&=&O(1)e^{-\frac{u^{2}}{2}+ O(u^{8\varepsilon})+\frac{A}{8\Delta ^{2}\sigma ^{2}}}\int_{\mathcal L_{u}}P(\mathcal{A}>\omega (u)) \\
&&\times \exp \Big\{-\frac{u\mathcal{A}}{\sigma }-\frac{1}{2}\frac{(\sqrt{A}w+\Delta
^{2}A^{-1/2}z)^{2}}{A-\Delta ^{2}}-\frac{1}{2}\Big[\frac{\Delta z}{\sqrt{A}}-
\frac{\sqrt{A}}{2\Delta }\Big(\frac{y^{2}}{\Delta u}+\frac{1}{\sigma }\Big)\Big]^{2}\Big\}dwdydz.
\end{eqnarray*}%
We introduce a  change of variable
\[
B=\frac{\Delta z}{\sqrt{A}}-\frac{\sqrt{A}}{2\Delta }(\frac{y^{2}}{\Delta u}+%
\frac{1}{\sigma }).
\]%
Then,%
\begin{eqnarray*}
\sqrt{A}w+\Delta ^{2}A^{-1/2}z &=&\Delta B+\sqrt{A}w+\frac{\sqrt{A}}{2}(%
\frac{y^{2}}{\Delta u}+\frac{1}{\sigma }) \\
&=&\frac{\sqrt{A}}{2\sigma }+\Delta B+\sqrt{A}\mathcal{A}+o(1).
\end{eqnarray*}%
Thus, by dominated convergence theorem and applying the change of variable from $(w,z,y)$ to $(\mathcal A, B, y)$, we have that
\begin{eqnarray}\label{1}
&&P\left( \max_{x\in \lbrack \frac L 2-u^{-1/2+\varepsilon },\frac L 2+u^{-1/2+\varepsilon
}]}|v^{\prime }(x)|>b;|y|\leq (1+\varepsilon ^{\prime })\Delta
u^{1/2+\varepsilon };\mathcal L_{u}\right)
=O(1)e^{-\frac{u^{2}}{2}+O(u^{8\varepsilon})}.
\end{eqnarray}

For $|y|>(1+\varepsilon ^{\prime })\Delta u^{1/2+\varepsilon }$, note that the function $|v'(x)|$ is maximized at $x= \frac L 2 +\gamma_{\ast }+\frac{y}{\Delta (u-z)}$, that is outside the interval $[ \frac L 2  -u^{-1/2+\varepsilon },\frac L 2 + u^{-1/2+\varepsilon }]$. Therefore, $\max_{x\in [ \frac L 2  -u^{-1/2+\varepsilon },\frac L 2 + u^{-1/2+\varepsilon }]}|v^{\prime}(x)|$ is less than  the estimate in \eqref{mmm} by at least a factor of $e^{-\lambda u^{2\varepsilon}}$ (by considering the dominating  term $\gamma e^{-\frac{\sigma \Delta u}{2}\gamma^{2}}$).
Therefore,
\[
\max_{x\in [ \frac L 2  -u^{-1/2+\varepsilon },\frac L 2 + u^{-1/2+\varepsilon }]}|v^{\prime}(x)|>b
\]%
if%
\[
\mathcal{A}=\sigma w+\frac{\sigma y^{2}}{2\Delta (u-z)}+\frac{z}{2u}-\frac{A%
}{4\Delta ^{3}u^{2}}y^{2}-\frac{\sigma A}{8\Delta ^{4}u^{3}}y^{4} + O(z^2/u^2)+O(u^{-1})
>\lambda
u^{2\varepsilon}+\omega (u).
\]%
Thus,%
\begin{eqnarray}\label{2}
&&P\Big( \max_{x\in [\frac L 2 -u^{-1/2+\varepsilon },\frac L 2 +u^{-1/2+\varepsilon
}]}|v^{\prime }(x)|>b;(1-\varepsilon ^{\prime })\Delta u^{1/2+\varepsilon
}\leq |y|\leq u^{1/2+4\varepsilon };\mathcal L_{u}\Big)   \\
&=&O(1)e^{-\frac{u^{2}}{2}+O(u^{8\varepsilon})}.  \nonumber
\end{eqnarray}%
We combine the solution of (\ref{1}), (\ref{2}), Lemma \ref{PropLocal1} and
obtain that%
\begin{eqnarray*}
\alpha (u,\varepsilon ) &=&P\left( \max_{x\in \lbrack\frac L 2  -u^{-1/2+\varepsilon
},\frac L 2 +u^{-1/2+\varepsilon }]}|v^{\prime }(x)|>b\right) \\
&=&O(1)e^{-\frac{u^{2}}{2}+O(u^{8\varepsilon})}.
\end{eqnarray*}%
Thus%
\[
P(E_{1})=O(1)u^{1/2-\varepsilon }\alpha (u,\varepsilon)=O(1)e^{-\frac{u^{2}}{2}+O(u^{8\varepsilon})}.
\]
As $\varepsilon$ can be chosen arbitrarily small, we obtain \eqref{e1} by redefining $\varepsilon$.

\section{Proofs of Propositions}

\begin{proof}[Proof of Proposition \protect\ref{PropLocal}]
The proof needs a change of measure described as follows. For $\zeta \in R$%
, let
$$A_{\zeta }=\{x :\xi (x)>\zeta \}\cap [ x_{\ast }+u^{-1/2+\delta/2},L-u^{-1/2+\delta }]$$
be the excursion set (on the
interval $[x_{\ast }+u^{-1/2+\delta /2},L-u^{-1/2+\delta }]$) over level $%
\zeta $ and let $P$ be the underlying nominal (original) probability
measure. Define $Q_{\zeta }\left( \cdot \right) $ via%
\begin{equation}
dQ_{\zeta }=\frac{mes(A_{\zeta })}{E(mes(A_{\zeta }))}dP=\frac{mes(A_{\zeta
})}{\int_{x_{\ast }+u^{-1/2+\delta /2}}^{L-u^{-1/2+\delta }}P(\xi (x)>\zeta
)dx}dP,  \label{measure}
\end{equation}%
where $E(\cdot )$ is the expectation under $P$ and $mes(A_{\zeta })$ is the
Lebesgue measure of the excursion set above level $\zeta $. Note that under $%
Q_{\zeta }$, almost surely $\sup_{L}\xi (x)>\zeta $. In order to generate
sample paths according $Q_{\zeta }$, one first simulates $\tau $ with
density function $\left\{ h\left( \tau \right) :\tau \in [ x_{\ast }+u^{-1/2+\delta/2},L-u^{-1/2+\delta }]\right\} $%
\begin{equation}
h(\tau )=\frac{P(\xi (\tau )>b)}{E(mes(A_{\zeta }))}  \label{DenTau}
\end{equation}%
that is a uniform distribution over the interval $[ x_{\ast }+u^{-1/2+\delta/2},L-u^{-1/2+\delta }]$;
then simulate $\xi (\tau )$ conditional distribution (under the original
law) given that $\xi (\tau )>\zeta $; lastly simulate $\{\xi (x):x\neq \tau
\}$ given $(\tau ,\xi (\tau ))$ according to the original distribution. If $%
\zeta $ is suitably chosen, $Q_{\zeta }$ serves as a good approximation of
the conditional distribution of $\xi (x)$ given that $\sup_{x\in [ x_{\ast }+u^{-1/2+\delta/2},L-u^{-1/2+\delta }]}\xi (x)>b$%
.

\begin{lemma}
\label{LemL1} Under conditions in Theorem \ref{ThmMain}, we have that%
\[
P\Big( \sup_{x\in \lbrack x_{\ast }+u^{-1/2+\delta /2},L-u^{-1/2+\delta
}]}\xi (x)>u-(\log u)^{2},E_{1}
\Big) =o(u^{-1}e^{-u^{2}/2}).
\]
\end{lemma}

\begin{proof}[Proof of Lemma \protect\ref{LemL1}]
Let%
\[
F_{b}=\{\sup_{x\in \lbrack x _{\ast }+u^{-1/2+\delta /2},L-u^{-1/2+\delta
}]}\xi (x)>u-(\log u)^{2}\}.
\]%
Let $\zeta =u-(\log u)^{2}-1/u$. Then, the probability can be written as
\begin{eqnarray*}
P\left( F_{b},E_{1}\right)  &\leq &O(1) E^{Q}\left[ \frac{P(Z>u-(\log u)^{2})}{%
mes(A_{\zeta })};F_{b},E_{1}\right]  \\
&=&O(1)\int_{x _{\ast }+u^{-1/2+\delta /2}}^{L-u^{-1/2+\delta }}E_{\tau
}^{Q}\left[ \frac{P(Z>u-(\log u)^{2})}{mes(A_{\zeta })};F_{b},E_{b}\right]
d\tau ,
\end{eqnarray*}%
where we use $E_{\tau }^{Q}$ to denote the conditional expectation $%
E^{Q}(\cdot |\tau )$ under the measure $Q_\zeta$. Given a particular $\tau \in \lbrack x_{\ast}+u^{-1/2+\delta /2},L-u^{-1/2+\delta }]$, we redefine the change of
variables
\[
\xi (\tau )=u+w,\xi ^{\prime }(\tau )=y,\xi ^{\prime \prime }(\tau )=-\Delta
(u-z).
\]
Note that the current definition of $(w,y,z)$ is different from that in the proposition and Theorem \ref{ThmMain}. As the previous definition of $(w,y,z)$ will not be used in this lemma, to simplify the notation, we do not create another notation and use $(w,y,z)$ differently.
Conditional on $(w,y,z)$ the process $g(x)$ is a mean zero Gaussian process such that
$$\xi(x) = E(\xi(x) | w,y,z) + g(x-\tau).$$
We have the bound of the excursion set that $E^{Q}(1/mes(A_{\zeta }))=O(u)$, the detailed development of which is omitted.
With this in mind, we first have that that
\[
E^{Q}\left[ \frac{P(Z>u-(\log u)^{2})}{mes(A_{\zeta })};|z|\geq
u^{1/2+\delta /16},F_{b},E_{b}\right] =o(u^{-1}e^{-u^{2}/2}).
\]%
and similarly%
\[
E^{Q}\left[ \frac{P(Z>u-(\log u)^{2})}{mes(A_{\zeta })};|y|\geq
u^{1/2+\delta /16},F_{b},E_{b}\right] =o(u^{-1}e^{-u^{2}/2}).
\]%
In addition, for some $\lambda _{0}$ sufficiently large and $\delta _{0}$
small, we have that
\begin{eqnarray*}
&&E\left( \frac{P(Z>u-(\log u)^{2})}{mes(A_{\zeta })};\sup_{|x|\leq
u^{-1/2+\delta }}|g(x)|>\lambda _{0}u^{-1+4\delta },\mbox{ or }%
\sup_{|x|>u^{-1/2+\delta }}|g(x)|-\delta _{0}ux^{2}>0\right)  \\
&=&o(u^{-1}e^{-u^{2}/2}).
\end{eqnarray*}%
Then, we only need to consider the situation that $|y|<u^{1/2+\delta /16}$
and $|z|<u^{1/2+\delta /16}$. Furthermore, using Taylor expansion on $\xi (x)
$ as we had done several times previously, the process $\xi (x)$ is a
approximately a quadratic function with mode being $\tau +\frac{y}{\Delta
\sigma (u-z)}$ for $\tau \in \lbrack x_{\ast }+u^{-1/2+\delta
/2},L-u^{-1/2+\delta }]$. Thus, when considering the integral $%
\int_{0}^{L}e^{\xi (t)}dt$ and $\int_{0}^{L}(F(x)-F(t))e^{\xi (t)}dt$, we do
not have to consider the boundary issue as in the analysis of $P(E_{2})$.
With the same calculations for \eqref{mmA} by expanding $\xi $ at $\tau $
instead of $x_{\ast }$, we obtain that%
\[
\sup_{x\in \lbrack u^{-1/2+\delta },L-u^{-1/2+\delta }]}|v^{\prime }(x)|\geq
b
\]%
if and only if%
\begin{eqnarray*}
\mathcal{A} &\mathcal{=}&\sigma w+\frac{\sigma y^{2}}{2\Delta (u-z)}+\frac{%
\sigma \Delta z}{2}\gamma _{\ast }^{2} \\
&&-\frac{\sigma A}{6\Delta }y(\gamma _{\ast }+\frac{y}{\Delta (u-z)})^{3}+%
\frac{\sigma Au}{24}(\gamma _{\ast }+\frac{y}{\Delta (u-z)})^{4} \\
&&-\frac{p^{\prime }(x)}{2p(x)\gamma _{\ast }}(\gamma _{\ast }^{2}+\frac{1}{%
\sigma \Delta (u-z)})+\frac{p^{\prime \prime }(x)}{6p(x)}(\gamma _{\ast
}^{2}+\frac{3}{\sigma \Delta (u-z)}) \\
&&-\frac{Ay^{3}}{\Delta ^{4}(u-z)^{3}\gamma _{\ast }}+\log \frac{p(x)}{%
p(x_{\ast })} \\
&\geq &o(u^{-1})+\omega (u),
\end{eqnarray*}%
where the $x$ in ``$p(x)$'' is $x= \tau + \gamma_* + \frac{y}{\Delta(u-z)} + o(u^{-1}) + O(z\gamma_*/u)$.
Similar to the derivation for \eqref{A}, we expand the second row in the definition of $\mathcal A$ and obtain that
\begin{eqnarray*}
\mathcal{A} &=&\sigma w+\frac{\sigma y^{2}}{2\Delta u}+\frac{\sigma }{%
2\Delta u^{2}}y^{2}z-\frac{\sigma Ay^{4}}{8\Delta ^{4}(u-z)^{3}}+\frac{%
\sigma \Delta z}{2}\gamma _{\ast }^{2}   -\frac{\sigma Ay^{2}}{4\Delta ^{2}(u-z)}\gamma _{\ast }^{2}+\frac{\sigma A(u-z)}{24}\gamma
_{\ast }^{4}  \\
&&-\frac{p^{\prime }(x)}{2p(x)\gamma_* }(\gamma _{\ast }^{2}+\frac{1}{\sigma
\Delta (u-z)})
  +\frac{p^{\prime \prime }(x_{\ast})}{6p(x_{\ast})}(\gamma _{\ast }^{2}+\frac{3}{\sigma \Delta (u-z)})  +\log \frac{p(x)}{p(x_*)}.\nonumber
\end{eqnarray*}%
Notice that
$$\frac{p^{\prime \prime }(x)}{6p(x)}(\gamma _{\ast }^{2}+\frac{3}{\sigma \Delta (u-z)}) = O(u^{-1}).$$
When $|x-x_{\ast }|< \varepsilon$, by Taylor expansion
$$|\frac{p^{\prime }(x)}{2p(x)\gamma_* }(\gamma _{\ast }^{2}+\frac{1}{\sigma
\Delta (u-z)})|=O((x-x_{\ast })/\sqrt{u}) = o(\log p(x) - \log p(x_*));$$
when $|x-x_{\ast }|> \varepsilon$
$$|\frac{p^{\prime }(x)}{2p(x)\gamma_* }(\gamma _{\ast }^{2}+\frac{1}{\sigma
\Delta (u-z)})|= O(u^{-1/2}) = o(1) = o(\log p(x) - \log p(x_*)).$$
Therefore $|\frac{p^{\prime }(x)}{2p(x)\gamma_* }(\gamma _{\ast }^{2}+\frac{1}{\sigma
\Delta (u-z)})|$ is always of a smaller order than $\log p(x) - \log p(x_*)$.
On the region $|x-x_{\ast }|>\frac{u^{-1/2+\delta /2}}{2}$,
there exists a positive $\lambda $ such that
\[
\log \frac{p(x)}{p(x_{\ast })}\leq -2\lambda u^{-1+\delta }.
\]%
Thus, $\mathcal{A}$ is bounded by
\begin{eqnarray*}
\mathcal{A}< \mathcal A' &=&\sigma w+\frac{\sigma y^{2}}{2\Delta u}
+\frac{\sigma }{2\Delta u^{2}}y^{2}z-\frac{\sigma Ay^{4}}{8\Delta ^{4}(u-z)^{3}}+\frac{\sigma \Delta z}{2}\gamma _{\ast }^{2}   -\frac{\sigma Ay^{2}}{4\Delta ^{2}(u-z)}\gamma _{\ast }^{2}+\frac{\sigma A(u-z)}{24}\gamma
_{\ast }^{4}  \\
&&-\lambda u^{-1+\delta}\nonumber
\end{eqnarray*}%
Furthermore, notice that
\begin{eqnarray*}
&&E_{\tau }^{Q}\left[ \frac{P(Z>u-(\log u)^{2})}{mes(A_{\zeta })}%
;|y|,|z|\leq u^{1/2+\delta /16},F_{b},E_{1}\right]  \\
&\leq &O(1)\int_{w\geq -(\log u)^{2}}e^{-\frac{1}{2}S(w,y,z)}\frac{P(%
\mathcal{A}^{\prime }\geq \omega (u),F_{b})}{mes(A_{\zeta })}dwdydz.
\end{eqnarray*}%
Similar to the previous development, we write
\begin{eqnarray*}
S(w,y,z) &=&u^{2}+w^{2}+\frac{\Delta ^{2}(w-z)^{2}}{A-\Delta ^{2}}+2u(w+%
\frac{y^{2}}{2\Delta u}) \\
&=&u^{2}+w^{2}+\frac{\Delta ^{2}(w-z)^{2}}{A-\Delta ^{2}} \\
&&+2u\Big[\frac{\mathcal{A}^{\prime }}{\sigma }
-\frac{y^{2}z }{2\Delta u^{2}}+\frac{ Ay^{4}}{8\Delta ^{4}(u-z)^{3}}-\frac{ \Delta z}{2}\gamma _{\ast }^{2}   +\frac{ Ay^{2}}{4\Delta ^{2}(u-z)}\gamma _{\ast }^{2}-\frac{ A(u-z)}{24}\gamma_{\ast }^{4}
+\lambda u^{-1+\delta }/\sigma\Big].
\end{eqnarray*}%
Thus, by dominated convergence theorem and the fact that $mes(A_{\zeta})^{-1} = O(u)$, we have that
\begin{eqnarray*}
&&E_{\tau }^{Q}\left[ \frac{P(Z>u-(\log u)^{2})}{mes(A_{\zeta })}%
;|y|,|z|\leq u^{1/2+\delta /16},F_{b},E_{1}\right]  \\
&\leq &O(1)\int_{|y|,|z|\leq u^{-1/2+\varepsilon /4}}
E(mes(A_{\zeta })^{-1};\mathcal{A}^{\prime }\geq \omega (u))
e^{-\frac{1}{2}S(w,y,z)}dwdydz \\
&\leq &O(1)e^{-\frac{u^{2}}{2}-\lambda u^{\delta }/\sigma}\\
&&\times \int_{|y|,|z|\leq
u^{-1/2+\varepsilon /4}}
E(mes(A_{\zeta })^{-1};\mathcal{A}^{\prime }\geq \omega (u))
\\
&&~~~~~~~~\times \exp \left[ -\frac{\Delta ^{2}}{2(A-\Delta ^{2})}z^{2}-\frac{u\mathcal{A}^{\prime }}{\sigma }
+\frac{y^{2}z }{2\Delta u}-\frac{ Ay^{4}}{8\Delta ^{4}u^{2}}+\frac{  z}{2\sigma}
+\frac{ Ay^{2}}{4\Delta ^{3}\sigma u}\right] dwdydz \\
&=&o(u^{-1}e^{-u^{2}/2}).
\end{eqnarray*}
\end{proof}

With a completely analogous proof as the Lemma \ref{LemL1}, we have that

\begin{lemma}
\label{LemL11} Under conditions in Theorem \ref{ThmMain}, we have that%
\[
P\left( \sup_{x\in \lbrack u^{-1/2+\delta },x_{\ast }-u^{-1/2+\delta
/2}]}\xi (x)>u-(\log u)^{2},E_{1}\right) =o(u^{-1}e^{-u^{2}/2}).
\]
\end{lemma}

\vskip12pt
We write
\[
J_{b}=\{\sup_{x\in \lbrack u^{-1/2+\delta },x_{\ast }-u^{-1/2+\delta
/2}]}\xi (x)>u-(\log u)^{2}\}\cup \{\sup_{x\in \lbrack \tau _{\ast
}+u^{-1/2+\delta /2},L-u^{-1/2+\delta }]}\xi (x)>u-(\log u)^{2}\}
\]%
and thus
$$P(J_b^c, E_1) = o(u^{-1}e^{-u^2/2}).$$
We proceed to the following lemma to complete the proof of the proposition.

\begin{lemma}
\label{LemL2} Let $(w,y,z)$ defined as in Section \ref{sec:3.1}. For $\varepsilon >0$, let%
\[
L_{b}=\{|w|<u^{3\delta },|y|<u^{1/2+4\delta },|z|<u^{1/2+4\delta }\}
\]%
Under conditions of Theorem \ref{ThmMain}, we have that%
\[
P\left( L_{b}^{c},J_{b}^{c},E_{1}\right) =o(u^{-1}e^{-u^{2}/2}).
\]
\end{lemma}

\begin{proof}
Note that $|v^{\prime }(x)|>b$ implies that $\xi (x)>\log b-\kappa
_{0}=u-O(\log u)$ for some $\kappa _{0}>0$. Thus, on the set $J_{b}^{c}$, $%
E_{1}$ implies that $\sup_{[x_{\ast }-u^{-1/2+\delta/2 },x_{\ast
+}u^{-1/2+\delta/2 }]}\xi (x)>\frac{\log b}{\sigma }-(\log u)^{2}$.
Therefore, we have that%
\[
P(|w|>u^{3\delta },F_{b}^{c},E_{b})\leq P(|w|>u^{3\delta },\sup_{[x_{\ast
}-u^{-1/2+\delta/2 },x_{\ast +}u^{-1/2+\delta/2 }]}\xi (x)>\frac{\log b%
}{\sigma }-(\log u)^{2})=o(u^{-1}e^{-u^{2}/2}),
\]%
where the last step is an application of Borel-TIS\ lemma. Furthermore, by
simply bound of Gaussian distribution, we have that%
\[
P(|w|<u^{3\delta },|z|>u^{1/2+4\delta
},F_{b}^{c},E_{b})=o(u^{-1}e^{-u^{2}/2}),
\]%
and%
\[
P(|w|<u^{3\delta },|y|>u^{1/2+4\delta
},F_{b}^{c},E_{b})=o(u^{-1}e^{-u^{2}/2}).
\]%
We thus conclude the proof.
\end{proof}

The results of Lemmas \ref{LemL1}, \ref{LemL11}, and \ref{LemL2} immediately
lead to the conclusion of Proposition \ref{PropLocal}.
\end{proof}

\bigskip

\begin{proof}[Proof of Proposition \protect\ref{PropG}]
Note that $g(x)$ is independent of $(w,y,z)$ and $\mathcal L_{u}$ only depends on $(w,y,z)$. Therefore,%
\begin{eqnarray*}
P\left( \sup_{|x|>u^{-1/2+8\delta }}
[|g(x)|-\delta ^{\prime
}ux^{2}]>0,
~\mathcal  L_{u}\right) &=&P\left( \sup_{|x|>u^{-1/2+8\delta }}[|g(x)|-\delta
^{\prime }ux^{2}]>0\right) P(\mathcal L_{u}) \\
&=&o(u^{-1}e^{-u^{2}/2}).
\end{eqnarray*}%
The last step is a direct application of the Borel-TIS lemma (Lemma \ref{LemBorel}) and the fact that $P(\mathcal L_u) = O(e^{-u^2/2 +O(u^{1+ 3\delta})})$. With a similar
argument, we obtain the second bound.
\end{proof}

\bigskip

\begin{proof}[Proof of Propositions \ref{PropLocal2} and \ref{PropLocal1}]
The proofs of these two propositions are completely analogous to that of Proposition  \ref{PropLocal}, that is, basically a repeated application of Borel-TIS lemma and the change of measure $Q_\zeta$. Therefore, we omit the details.
\end{proof}

\section{Proof of the Lemmas}

\begin{proof}[Proof of Lemma \protect\ref{LemInt}]
On the set $|x-x_{\ast}|<u^{-1/2+8\delta }$ and $\mathcal L_{u}^{\prime }$, we have $%
s=O(u^{8\delta })$ and thus
\[
\frac{y^{3}s}{(u-z)^{5/2}}=O(u^{-1 +20\delta }),\frac{y^{2}s^{2}%
}{(u-z)^{2}}=O(u^{-1+24\delta }),\frac{s^{4}}{(u-z)}%
=O(u^{-1+32\delta }).
\]%
Let $X$ be a standard Gaussian random variable. We conclude the proof by the
following calcuation%
\begin{eqnarray*}
&&\int_{|x-x_{\ast}|<u^{-1/2+8\delta }}e^{\sigma \lbrack -\frac{s^{2}}{2}-%
\frac{Ay^{3}}{\Delta ^{7/2}(u-z)^{5/2}}s-\frac{Ay^{2}}{4\Delta ^{3}(u-z)^{2}}%
s^{2}+\frac{A}{24\Delta ^{2}(u-z)}s^{4}]} ds \\
&=&e^{o(u^{-1})}\int_{|x-x_{\ast}|<u^{-1/2+8\delta }}e^{-\frac{\sigma s^{2}}{%
2}}
\times \left(1-\frac{\sigma Ay^{3}}{\Delta ^{7/2}(u-z)^{5/2}}s-\frac{\sigma Ay^{2}}{%
4\Delta ^{3}(u-z)^{2}}s^{2}+\frac{\sigma A}{24\Delta ^{2}(u-z)}s^{4}\right)ds \\
&=&e^{o(u^{-1})}\sqrt{\frac{2\pi }{\sigma }}E\left[1-\frac{A\sigma ^{1/2}y^{3}X}{%
\Delta ^{7/2}(u-z)^{5/2}}-\frac{Ay^{2}X^{2}}{4\Delta ^{3}(u-z)^{2}}+\frac{%
AX^{4}}{24\Delta ^{2}\sigma (u-z)}\right] \\
&=&\sqrt{\frac{2\pi }{\sigma }}\exp \left\{ -\frac{Ay^{2}}{4\Delta
^{3}(u-z)^{2}}+\frac{A}{8\Delta ^{2}\sigma (u-z)}+o(u^{-1})\right\} \\
&=&\sqrt{\frac{2\pi }{\sigma }}\exp \left\{ -\frac{Ay^{2}}{4\Delta
^{3}(u-z)^{2}}+\frac{A}{8\Delta ^{2}\sigma u}+o(u^{-1})\right\} .
\end{eqnarray*}
\end{proof}

\bigskip

\begin{proof}[Proof of Lemma \ref{Lemf1}]
We use the result of Lemma \ref{LemInt} and the Taylor expansion%
\[
F(x)-F(t)=p(x)(x-t)-\frac{1}{2}p^{\prime }(x)(x-t)^{2}+\frac{1}{6}p^{\prime
\prime }(x)(x-t)^{3}+o(x-t)^{4}.
\]%
Recall the change of variable
$$s(t)= \sqrt{\Delta (u-z)}\left( t-x_{\ast}-\frac{y}{\Delta (u-z)}\right)$$
as in \eqref{eqn:schange}. We apply it to the spatial index $t$.
Note that $t-x_{\ast}-s(t)/\sqrt{\Delta (u-z)}=y/(\Delta (u-z))$ and $%
x-t=\gamma -s(t)/\sqrt{\Delta (u-z)}$.
 We perform the same splitting as in (\ref{split}), insert the result in \eqref{deno}, use the expansion of $\xi$ in \eqref{xi}, and obtain that
\begin{eqnarray*}
&&\left( \int_{0}^{L}e^{\sigma \xi (t)}dt\right)
^{-1}\int_{0}^{L}(F(x)-F(t))e^{\sigma \xi (t)}dt \\
&=&\exp \left\{ \frac{Ay^{2}}{4\Delta ^{3}(u-z)^{2}}-\frac{A}{8\Delta
^{2}\sigma (u-z)}+\omega (u)+o(u^{-1})\right\} \\
&&\times
 \int_{|s|\leq u^{8\delta}}
 \bigg [
 p(x)\Big (\gamma -\frac{s}{\sqrt{\Delta (u-z)}%
}\Big)- \frac{1}{2}p^{\prime }(x)
\Big(\gamma -\frac{s}{\sqrt{\Delta (u-z)}}\Big)^{2} \\
&&\quad\quad \quad \quad~~~ +\frac{1}{6}p^{\prime \prime }(x)
\Big(\gamma -\frac{s}{\sqrt{\Delta (u-z)}}\Big)^{3}
+o(u^{-3/2})
\bigg] \\
&&\qquad \times  \sqrt{\frac{\sigma }{2\pi }}e^{\sigma \lbrack -\frac{s^{2}}{2}-
\frac{Ay^{3}}{3\Delta ^{7/2}(u-z)^{5/2}}s-\frac{Ay^{2}}{4\Delta ^{3}(u-z)^{2}%
}s^{2}+\frac{A}{24\Delta ^{2}(u-z)}s^{4}]}ds
\end{eqnarray*}%
We rewrite the above integral by pulling out the Gaussian density and
expanding the exponential term in the last row
\begin{eqnarray*}
&=&\exp \left\{\frac{Ay^{2}}{4\Delta ^{3}(u-z)^{2}}-\frac{A}{8\Delta
^{2}\sigma (u-z)}+\omega (u)+o(u^{-1})\right\} \\
&&\times
\int _{|s|\leq u^{8\delta}|}\sqrt{\frac{\sigma }{2\pi }}e^{-\frac{\sigma s^{2}}{2}} \\
&&\quad \times \left[ p(x)\Big(\gamma -\frac{s}{\sqrt{\Delta (u-z)}}\Big)-\frac{1}{2}%
p^{\prime }(x)\Big(\gamma -\frac{s}{\sqrt{\Delta (u-z)}}\Big)^{2}+\frac{1}{6}%
p^{\prime \prime }(x)\Big(\gamma -\frac{s}{\sqrt{\Delta (u-z)}}\Big)^{3}\right] \\
&&\quad \times \left[ 1-\frac{\sigma Ay^{3}}{3\Delta ^{7/2}(u-z)^{5/2}}s-\frac{%
\sigma Ay^{2}}{4\Delta ^{3}(u-z)^{2}}s^{2}+\frac{\sigma A}{24\Delta ^{2}(u-z)%
}s^{4}\right] ds.
\end{eqnarray*}%
Similar to Lemma \ref{LemInt}, we further evaluate the above integral by computing moments of $N(0,\sigma
^{-1/2})$ and obtain that (we omit several cross terms that can be absorbed  by $o(u^{-1})$)
\begin{eqnarray*}
&&F(x)-\frac{\int_{0}^{L}F(t)e^{\sigma \xi (t)}dt}{\int_{0}^{L}e^{\sigma \xi
(t)}dt} \\
&=&\exp \left\{ \frac{Ay^{2}}{4\Delta ^{3}(u-z)^{2}}-\frac{A}{8\Delta
^{2}\sigma (u-z)}+\omega (u)+o(u^{-1})\right\}    \\
&&\times
\bigg \lbrack p(x)\gamma -\frac{p^{\prime }(x)}{2}
\left(\gamma ^{2}+\frac{1}{\sigma \Delta (u-z)}\right) +\frac{p^{\prime \prime }(x)}{6}\left(\gamma ^{3}+\frac{3\gamma }{\sigma \Delta (u-z)}\right) \\
&&~~~~~+p(x)\frac{Ay^{3}}{3\Delta ^{4}(u-z)^{3}}-p(x)\gamma \frac{Ay^{2}}{4\Delta
^{3}(u-z)^{2}}+p(x)\gamma \frac{A}{8\sigma \Delta ^{2}(u-z)}
\bigg ].
\end{eqnarray*}%
We take out the factor ``$p(x)\gamma $'' from the bracket and continue the calculation
\begin{eqnarray*}
&=&\exp \left\{ \frac{Ay^{2}}{4\Delta ^{3}(u-z)^{2}}-\frac{A}{8\Delta
^{2}\sigma (u-z)}+\omega (u)+o(u^{-1})\right\}    \\
&&\times p(x)\gamma \exp \Big[-\frac{p^{\prime }(x)}{2p(x)\gamma }(\gamma ^{2}+\frac{1%
}{\sigma \Delta (u-z)})+\frac{p^{\prime \prime }(x)}{6p(x)}(\gamma ^{2}+%
\frac{3}{\sigma \Delta (u-z)}) \notag\\
&&~~~~~~~~~~~~~~~~~~+\frac{Ay^{3}}{3\Delta ^{4}(u-z)^{3}\gamma }- \frac{Ay^{2}}{4\Delta ^{3}(u-z)^{2}}+ \frac{A}{8\sigma \Delta ^{2}(u-z)}\Big].
\end{eqnarray*}%
We further simplify the above display and obtain that
\begin{eqnarray*}
&=&p(x)\gamma \exp \Big[-\frac{p^{\prime }(x)}{2p(x)\gamma }(\gamma ^{2}+\frac{1%
}{\sigma \Delta (u-z)})+\frac{p^{\prime \prime }(x)}{6p(x)}(\gamma ^{2}+%
\frac{3}{\sigma \Delta (u-z)}) \notag\\
&&~~~~~~~~~~~~~~~~~~+\frac{Ay^{3}}{3\Delta ^{4}(u-z)^{3}\gamma }+o(u^{-1})+\omega (u)\Big].
\end{eqnarray*}
\end{proof}

\bigskip

\begin{proof}[Proof of Lemma \ref{mmmA}]
Let $\mathcal A$ be defined as in \eqref{mmA}.
Note that $p'(x_*) = 0$ and $p^{\prime }(x)\sim p''(x_{\ast})(\gamma +y/\Delta (u-z))$.
We apply Taylor expansion of the term $\log \frac{p(x_{\ast} + \gamma_*+\Delta ^{-1}(u-z)^{-1}y)}{p(x_{\ast})} $ in  \eqref{mmA} and expand the second row of \eqref{mmA}.
Thus, $\mathcal{A}$ can be further simplified to
\begin{eqnarray*}
\mathcal{A} &\mathcal{=}&\sigma w+\frac{\sigma y^{2}}{2\Delta u}+\frac{%
\sigma }{2\Delta u^{2}}y^{2}z-\frac{\sigma Ay^{4}}{8\Delta ^{4}(u-z)^{3}}+%
\frac{\sigma \Delta z}{2}\gamma _{\ast }^{2} \\
&&-\frac{\sigma Ay^{3}}{3\Delta ^{3}(u-z)^{2}}\gamma _{\ast }-\frac{\sigma
Ay^{2}}{4\Delta ^{2}(u-z)}\gamma _{\ast }^{2}+\frac{\sigma Au}{24}\gamma
_{\ast }^{4} \\
&&-\frac{p^{\prime \prime }(x_{\ast})}{2p(x_{\ast})}(\gamma _{\ast }+\frac{%
y}{\Delta (u-z)})(\gamma _{\ast }+\frac{1}{\sigma \Delta (u-z)\gamma _{\ast }%
}) \\
&&+\frac{p^{\prime \prime }(x_{\ast})}{6p(x_{\ast})}(\gamma _{\ast }^{2}+%
\frac{3}{\sigma \Delta (u-z)})+\frac{Ay^{3}}{3\Delta ^{4}(u-z)^{3}\gamma
_{\ast }} \\
&&+\frac{p^{\prime \prime }(x_{\ast})}{2p(x_{\ast})}(\gamma _{\ast }+\frac{%
y}{\Delta (u-z)})^{2} + o(y^2 u^{-2})+O(z^2 /u^2).
\end{eqnarray*}%
Note that $\gamma_* = u^{-1/2}\Delta ^{-1/2}\sigma ^{-1/2}$. The term
$$-\frac{p^{\prime \prime }(x_{\ast})}{2p(x_{\ast})}\frac{y}{\Delta (u-z)}(\gamma _{\ast }+\frac{1}{\sigma \Delta (u-z)\gamma _{\ast }%
})$$
expanded from the third row cancels the cross term
$$\frac{\gamma _* p^{\prime \prime }(x_{\ast})}{p(x_{\ast})}\frac{y}{\Delta (u-z)}$$
expanded from the quadratic term in the last row.
Then, $\mathcal A$ is further simplified to
\begin{eqnarray}
\mathcal{A} &=&\sigma w+\frac{\sigma y^{2}}{2\Delta u}+\frac{\sigma }{%
2\Delta u^{2}}y^{2}z-\frac{\sigma Ay^{4}}{8\Delta ^{4}(u-z)^{3}}+\frac{%
\sigma \Delta z}{2}\gamma _{\ast }^{2}  \label{A} \\
&&-\frac{\sigma Ay^{3}}{3\Delta ^{3}(u-z)^{2}}\gamma _{\ast }-\frac{\sigma
Ay^{2}}{4\Delta ^{2}(u-z)}\gamma _{\ast }^{2}+\frac{\sigma A(u-z)}{24}\gamma
_{\ast }^{4}  \nonumber \\
&&-\frac{p^{\prime \prime }(x_{\ast})}{2p(x_{\ast})}(\gamma _{\ast }^{2}+%
\frac{1}{\sigma \Delta (u-z)})  \nonumber \\
&&+\frac{p^{\prime \prime }(x_{\ast})}{6p(x_{\ast})}(\gamma _{\ast }^{2}+%
\frac{3}{\sigma \Delta (u-z)})+\frac{Ay^{3}}{3\Delta ^{4}(u-z)^{3}\gamma
_{\ast }}  \nonumber \\
&&+\frac{p^{\prime \prime }(x_{\ast})}{2p(x_{\ast})}(\gamma _{\ast }^{2}+%
\frac{y^{2}}{\Delta ^{2}(u-z)^{2}})+o(y^2 u^{-2})+O(z^2 /u^2).  \nonumber
\end{eqnarray}%
Furthermore, the term $-\frac{\sigma Ay^{3}}{3\Delta ^{3}(u-z)^{2}}\gamma _{\ast }$ in the second row cancels $\frac{Ay^{3}}{3\Delta ^{4}(u-z)^{3}\gamma_{\ast }}$ in the fourth row.
We now plug in $\gamma _{\ast }^{2}=\Delta ^{-1}\sigma ^{-1}u^{-1}$ and obtain
that%
\begin{eqnarray*}
\mathcal{A} &=&\sigma w+\frac{\sigma y^{2}}{2\Delta u}+\frac{\sigma }{%
2\Delta u^{2}}y^{2}z-\frac{\sigma Ay^{4}}{8\Delta ^{4}u^{3}}+\frac{z}{2u} \\
&&-\frac{Ay^{2}}{4\Delta ^{3}u^{2}}+\frac{A}{24\sigma \Delta ^{2}u}-%
\frac{p^{\prime \prime }(x_{\ast})}{3p(x_{\ast})\sigma \Delta u}+\frac{%
p^{\prime \prime }(x _{\ast })}{2p(x _{\ast })}(\frac{1}{\sigma
\Delta u}+\frac{y^{2}}{\Delta ^{2}u^{2}})+o(u^{-1}) +O(z^2/u)\\
&=&\sigma w+\frac{\sigma y^{2}}{2\Delta u}+\frac{\sigma }{2\Delta u^{2}}%
y^{2}z+\frac{z}{2u}+\frac{A}{24\sigma \Delta ^{2}u}+\frac{p^{\prime \prime }(x_{\ast})}{6p(x_{\ast})\sigma \Delta u} \\
&&-\frac{\sigma Ay^{4}}{8\Delta u^{3}}+\frac{y^{2}}{u^{2}}(-\frac{A}{4\Delta
^{3}}+\frac{p^{\prime \prime }(x_{\ast})}{2p(x_{\ast})\Delta ^{2}}%
)+o(u^{-1}+y^2 u^{-2})+O(z^2 /u^2).
\end{eqnarray*}
\end{proof}

\bigskip

\begin{proof}[Proof of Lemma \ref{Lemf2}]
Using the second change of variable in \eqref{change}, the denominator in \eqref{f2} is
\begin{eqnarray*}
\int_{0}^{L}e^{\sigma \xi(t)}dt &=&e^{c_{\ast }}\int_{0}^{L}\exp \Big\{\sigma
\Big[ -\frac{s^{2}}{2}-\frac{Ay^{3}}{3\Delta ^{7/2}u_{L}^{5/2}}s-\frac{%
Ay^{2}}{4\Delta ^{3}u_{L}^{2}}s^{2}+\frac{A}{24\Delta ^{2}u_{L}}s^{4}\Big]\Big\}dt .
\end{eqnarray*}
Let $Z$ be a standard Gaussian random variable following $N(0,1)$. With a similar splitting in \eqref{split} and the derivation in Lemma \ref{LemInt} and noticing the boundary constraint that \eqref{boundary}, we apply Taylor expansion on the integrand and have that
\begin{eqnarray*}
&=&\frac{\sqrt{2\pi }e^{c_{\ast }+o(u_{L}^{-1})}}{\sqrt{\Delta \sigma
(u_{L}-z)}}e^{\omega (u_{L})} \\
&&\times E\left[ 1-\frac{\sigma ^{1/2}Ay^{3}}{3\Delta ^{7/2}u_{L}^{5/2}}Z-%
\frac{Ay^{2}}{4\Delta ^{3}u_{L}^{2}}Z^{2}+\frac{A}{24\Delta ^{2}\sigma u_{L}}%
Z^{4};Z\leq \sqrt{1-\frac{z}{u_{L}}}\zeta _{L}-\sqrt{\frac{\sigma }{\Delta
(u_{L}-z)}}y\right] \\
&=&\frac{\sqrt{2\pi }e^{c_{\ast }+o(u_{L}^{-1})}}{\sqrt{\Delta \sigma
(u_{L}-z)}}e^{\omega (u_{L})+O(y^{3}/u_{L}^{5/2}+y^{2}/u_{L}^{2})} \\
&&\times E\left[ 1+\frac{A}{24\Delta ^{2}\sigma u_{L}}Z^{4};Z\leq \sqrt{1-%
\frac{z}{u_{L}}}\zeta _{L}-\sqrt{\frac{\sigma }{\Delta (u_{L}-z)}}y\right],
\end{eqnarray*}
where $c_{\ast }=\sigma (u_{L}+w+\frac{y^{2}}{2\Delta (u_{L}-z)}-\frac{Ay^{4}}{8\Delta ^{4}(u_{L}-z)^{3}})$ and $\omega (u)=O(\sup_{|x|\leq u^{-1/2+8\delta }}|g(x)|)$.
The expectation in the previous display can be written as
\begin{eqnarray*}
&&E\left[ 1+\frac{A}{24\Delta ^{2}\sigma u_{L}}Z^{4};Z\leq \sqrt{1-\frac{z}{u_{L}}}\zeta _{L}-\sqrt{\frac{\sigma }{\Delta (u_{L}-z)}}y\right]\\
&=&P\Big[Z\leq \sqrt{1-\frac{z}{u_{L}}}\zeta _{L}-\sqrt{\frac{\sigma }{%
\Delta (u_{L}-z)}}y\Big] \\
&&\times \exp \left\{ \frac{A}{24\Delta ^{2}\sigma u_{L}}E(Z^{4}|Z\leq \zeta
_{L})+\omega
(u_{L})+O(y^{3}/u_{L}^{5/2}+y^{2}/u_{L}^{2} + y /u^{3/2})\right\}
\end{eqnarray*}
We use the fact that  $E(Z^{4}|Z\leq \zeta_{L}) = E(Z^{4}|Z\leq\sqrt{1-\frac{z}{u_{L}}}\zeta _{L}-\sqrt{\frac{\sigma }{%
\Delta (u_{L}-z)}}y) + o(1+ yu^{-1/2}).$
We continue the calculations and obtain that
\begin{eqnarray*}
\int_{0}^{L}e^{\sigma \xi(t)}dt&=&\frac{\sqrt{2\pi }e^{c_{\ast }+o(u_{L}^{-1})}}{\sqrt{\Delta \sigma
(u_{L}-z)}}P\Big[Z\leq \sqrt{1-\frac{z}{u_{L}}}\zeta _{L}-\sqrt{\frac{\sigma }{%
\Delta (u_{L}-z)}}y\Big] \\
&&\times \exp \left\{ \frac{A}{24\Delta ^{2}\sigma u_{L}}E(Z^{4}|Z\leq \zeta
_{L})+\omega
(u_{L})+O(y^{3}/u_{L}^{5/2}+y^{2}/u_{L}^{2} + y /u^{3/2})\right\}.
\end{eqnarray*}%
We now proceed to the numerator of \eqref{f2}. Using  Taylor expansion
\[
F(x)-F(t)=p(x)(x-t)-\frac{1}{2}p^{\prime }(x)(x-t)^{2}+\frac{1}{6}p^{\prime
\prime }(x)(x-t)^{3}+o(x-t)^{3},
\]%
the numerator of \eqref{f2} is (with the splitting as in \eqref{split})
\begin{eqnarray*}
	&&\int_{0}^{L}(F(x)-F(t))e^{\sigma \xi(t)}dt \\
	&=&\frac{e^{c_{\ast }+\omega(u_{L})+o(u_{L}^{-1})}}{\sqrt{\Delta (u_{L}-z)}}\times \int_{-u^{8\delta}}^{\sqrt{\frac{(1-z/u)}{\sigma }}\zeta _{L}-\frac{y}{\sqrt{\Delta (u-z)}}} \\
	&&\left[ p(x)(\gamma -\frac{s}{\sqrt{\Delta (u_{L}-z)}})-\frac{1}{2}%
p^{\prime }(x)(\gamma -\frac{s}{\sqrt{\Delta (u_{L}-z)}})^{2}+\frac{1}{6}%
p^{\prime \prime }(x)(\gamma -\frac{s}{\sqrt{\Delta (u_{L}-z)}}%
)^{3}+o(u_{L}^{-3/2})\right] \\
	&&\times e^{\sigma \Big \{ -\frac{s^{2}}{2}-\frac{Ay^{3}}{3\Delta
^{7/2}(u_{L}-z)^{5/2}}s-\frac{Ay^{2}}{4\Delta ^{3}(u_{L}-z)^{2}}s^{2}+\frac{A%
}{24\Delta ^{2}(u_{L}-z)}s^{4}\Big \}}ds \\
	&=&\sqrt{\frac{2\pi }{\Delta \sigma (u_{L}-z)}}e^{c_{\ast }+\omega
(u_{L})+o(u_{L}^{-1})} \\
	&&\times E\Big\{p(x)(\gamma -\frac{Z}{\sqrt{\Delta \sigma (u_{L}-z)}})-\frac{%
p^{\prime }(x)}{2}(\gamma -\frac{Z}{\sqrt{\Delta \sigma (u_{L}-z)}})^{2}+%
\frac{p^{\prime \prime }(x)}{6}(\gamma -\frac{Z}{\sqrt{\Delta \sigma
(u_{L}-z)}})^{3} \\
	&&~~~~~+\frac{A p(x)}{24\Delta ^{2}\sigma ^{2}u_{L}}Z^{4}(\gamma -\frac{Z}{\sqrt{%
\Delta \sigma (u_{L}-z)}})\\
&&~~~~~ +O(y^{3}/u_{L}^{5/2}+y^{2}/u_{L}^{2}+u_{L}^{-2})~~; ~~Z\leq \sqrt{1-
\frac{z}{u_{L}}}\zeta _{L}-\sqrt{\frac{\sigma }{\Delta (u_{L}-z)}}y\Big\}.
\end{eqnarray*}%
Thus, 
the factor in (\ref{f2})\ is
\begin{eqnarray*}
&&\int_{0}^{L}(F(x)-F(t))\frac{e^{\sigma \xi(t)}}{\int_{0}^{L}e^{\sigma \xi(s)}ds%
}dt \\
&=&\exp \Big\{-\frac{A}{24\Delta ^{2}\sigma u_{L}}E\left( Z^{4}|Z\leq \zeta
_{L}\right) +\lambda (u_{L})+\omega (u_{L})\Big\} \\
&&\times E \Big\{ p(x)(\gamma -\frac{Z}{\sqrt{\Delta \sigma (u_{L}-z)}})-\frac{%
p^{\prime }(x)}{2}(\gamma -\frac{Z}{\sqrt{\Delta \sigma (u_{L}-z)}})^{2}+%
\frac{p^{\prime \prime }(x)}{6}(\gamma -\frac{Z}{\sqrt{\Delta \sigma
(u_{L}-z)}})^{3} \\
&& ~~~~~~+\frac{Ap(x)}{24\Delta ^{2}\sigma ^{2}u_{L}}Z^{4}(\gamma -\frac{Z}{\sqrt{%
\Delta \sigma (u_{L}-z)}})~~\Big\vert~~ Z\leq \sqrt{1-\frac{z}{u_{L}}}\zeta
_{L}-\sqrt{\frac{\sigma }{\Delta (u_{L}-z)}}y \Big\}
\end{eqnarray*}%
where $\lambda (u_{L})=O(y^{3}/u_{L}^{5/2}+y^{2}/u_{L}^{2} + y /u^{3/2})+o(u_{L}^{-1}+u_{L}^{-1}z)$. We take
out a factor $\sqrt{\Delta \sigma (u_{L}-z)}$ from the above expectation and
obtain that%
\begin{eqnarray*}
&=&\exp \Big\{-\frac{A}{24\Delta ^{2}\sigma u_{L}}E\left( Z^{4}|Z\leq \zeta
_{L}\right) +\lambda (u_{L})+\omega (u_{L})\Big \} \\
&&\frac{1}{\sqrt{\Delta \sigma u_{L}(1-z/u_{L})}}E\Big\{p(x)(\gamma \sqrt{\sigma
\Delta (u_{L}-z)}-Z)
-\frac{p'(x)}{2\sqrt{\sigma \Delta u_{L}}}(\gamma \sqrt{\sigma \Delta (u_{L}-z)}-Z)^{2} \\
&&+\frac{p^{\prime \prime }(x)}{6\sigma \Delta u_{L}}(\gamma \sqrt{\sigma
\Delta u_{L}}-Z)^{3}+\frac{Ap(x)}{24\Delta ^{2}\sigma ^{2}u_{L}}Z^{4}(\gamma
\sqrt{\sigma \Delta u_{L}}-Z)~~\Big\vert~~ Z\leq \sqrt{1-\frac{z}{u_{L}}}\zeta
_{L}-\sqrt{\frac{\sigma }{\Delta (u_{L}-z)}}y \Big\}.
\end{eqnarray*}%
Notice that in the last two terms of the above display and for the denominator of the second term in the second low, ``$u_L - z$'' is replaced by $u_L$. The error caused by this change can be absorbed into $\lambda (u_L)$.
Notice that
$$\frac{1}{\sqrt{\Delta \sigma u_{L}(1-z/u_{L})}} = \frac{e^{\frac{z}{2u_{L}}+ o(z/u_L)} }{\sqrt{\Delta \sigma u_{L}}}.$$
We further separate the expectation into two parts and obtain that%
\begin{eqnarray*}
&=&\exp \Big\{-\frac{A}{24\Delta ^{2}\sigma u_{L}}E\left( Z^{4}|Z\leq \zeta
_{L}\right) +\lambda (u_{L})+\omega (u_{L})\Big\} \times \frac{e^{\frac{z}{2u_{L}}} }{\sqrt{\Delta \sigma u_{L}}}\\
&&\times \Big\{E\Big[p(x)(\gamma \sqrt{\sigma \Delta (u_{L}-z)}-Z)\\
&&~~~~~~~~~~-\frac{p'(x)}{2\sqrt{\sigma \Delta u_{L}}}(\gamma \sqrt{\sigma \Delta (u_{L}-z)}-Z)^{2} ~~\Big\vert~~ Z\leq \sqrt{1-\frac{z}{u_{L}}}\zeta _{L}-\sqrt{\frac{\sigma }{\Delta (u_{L}-z)}}y \Big] \\
&&~~~+E\Big[\frac{p^{\prime
\prime }(x)}{6\sigma \Delta u_{L}}(\gamma \sqrt{\sigma \Delta u_{L}}-Z)^{3}\\
&&~~~~~~~~~~+\frac{Ap(x)}{24\Delta ^{2}\sigma ^{2}u_{L}}Z^{4}(\gamma \sqrt{\sigma
\Delta u_{L}}-Z)~~\Big \vert ~~Z\leq \sqrt{1-\frac{z}{u_{L}}}\zeta _{L}-\sqrt{%
\frac{\sigma }{\Delta (u_{L}-z)}}y \Big]\Big\}.
\end{eqnarray*}%
Thus, we conclude the proof.
\end{proof}

\bigskip

\begin{proof}[Proof of Lemma \ref{Lemvprime}]
Similar to the calculations resulting \eqref{xi1}, we obtain that
\begin{eqnarray*}
\xi (x) &=&u_{L}+w+\frac{y^{2}}{2\Delta (u_{L}-z)}-\frac{\Delta (u_{L}-z)}{2}%
\gamma ^{2}-\frac{A}{6\Delta }y(\gamma +\frac{y}{\Delta (u_{L}-z)})^{3}+%
\frac{Au_{L}}{24}(\gamma +\frac{y}{\Delta (u_{L}-z)})^{4} \\
&&+g(x-t_{L})+\vartheta (x-t_{L}) \\
&=&u_{L}+w+\frac{y^{2}}{2\Delta u_{L}}-\frac{\Delta (u_{L}-z)}{2}\gamma ^{2}
+\frac{Au_{L}}{24}\gamma
^{4}+o(u^{-1}y^{2})+g(x-t_{L})+\vartheta (x-t_{L}),
\end{eqnarray*}%
where $\vartheta(x)=O(u^{1/2+4\delta }x^{5}+ux^{6})$.
Combining the above expression and Lemma \ref{Lemf2}, we obtain that
\begin{eqnarray*}
v^{\prime }(x) &=&e^{\sigma \xi(x)}\int_0^L (F(x)-F(t))\frac{e^{\sigma \xi(t)}}{%
\int e^{\sigma \xi(s)}ds}dt \\
&=&\exp \Big\{\lambda (u_{L}) + O(y^2z u_L^{-2})+\omega (u_{L})+\sigma u_{L}+\sigma w+\frac{\sigma
y^{2}}{2\Delta u_{L}}+\frac{A\sigma
u_{L}}{24}\gamma ^{4}\Big\} \\
&&\times \frac{1}{\sqrt{\Delta \sigma u_{L}}}\exp \Big\{-\frac{\sigma \Delta (u_{L}-z)}{2}\gamma ^{2}+\frac{z}{2u_{L}}-\frac{A}{24\Delta ^{2}\sigma u_{L}}E\left(
Z^{4}|Z\leq \zeta _{L}\right) \Big\} \\
&&\Big\{E\Big[p(x)(\gamma \sqrt{\sigma \Delta (u_{L}-z)}-Z)\\
&&~~~~~~~~-\frac{p^{\prime }(x)}{2\sqrt{\sigma \Delta u_{L}}}(\gamma \sqrt{\sigma \Delta (u_{L}-z)}-Z)^{2}~~\Big\vert~~
Z\leq \sqrt{1-\frac{z}{u_{L}}}\zeta _{L}-\sqrt{\frac{\sigma }{\Delta
(u_{L}-z)}}y \Big] \\
&&~~+E\Big[\frac{p''(x)}{6\sigma \Delta u_{L}}(\gamma \sqrt{\sigma \Delta u_{L}}-Z)^{3} \\
&&~~~~~~~~~~~~+\frac{Ap(x)}{24\Delta ^{2}\sigma ^{2}u_{L}}Z^{4}(\gamma \sqrt{\sigma \Delta
u_{L}}-Z)~~\Big\vert~~ Z\leq \sqrt{1-\frac{z}{u_{L}}}\zeta _{L}-\sqrt{\frac{%
\sigma }{\Delta (u_{L}-z)}}y\Big]\Big\}
\end{eqnarray*}
Using Taylor expansion on the two expectation terms, we obtain that
\begin{eqnarray*}
&&E\Big[p(x)(\gamma \sqrt{\sigma \Delta (u_{L}-z)}-Z)\\
&&~~~~~~~~-\frac{p^{\prime }(x)}{2\sqrt{\sigma \Delta u_{L}}}(\gamma \sqrt{\sigma \Delta (u_{L}-z)}-Z)^{2}~~\Big\vert~~
Z\leq \sqrt{1-\frac{z}{u_{L}}}\zeta _{L}-\sqrt{\frac{\sigma }{\Delta
(u_{L}-z)}}y \Big] \\
&&~~+E\Big[\frac{p''(x)}{6\sigma \Delta u_{L}}(\gamma \sqrt{\sigma \Delta u_{L}}-Z)^{3} \\
&&~~~~~~~~~~~~+\frac{Ap(x)}{24\Delta ^{2}\sigma ^{2}u_{L}}Z^{4}(\gamma \sqrt{\sigma \Delta
u_{L}}-Z)~~\Big\vert~~ Z\leq \sqrt{1-\frac{z}{u_{L}}}\zeta _{L}-\sqrt{\frac{%
\sigma }{\Delta (u_{L}-z)}}y\Big]\\
&=&E\Big[p(x)(\gamma \sqrt{\sigma \Delta (u_{L}-z)}-Z)\\
&&~~~~~~~~-\frac{p^{\prime }(x)}{2\sqrt{\sigma \Delta u_{L}}}(\gamma \sqrt{\sigma \Delta (u_{L}-z)}-Z)^{2}~~\Big\vert~~
Z\leq \sqrt{1-\frac{z}{u_{L}}}\zeta _{L}-\sqrt{\frac{\sigma }{\Delta
(u_{L}-z)}}y \Big]\\
&&\times \exp \left\{ \frac{E\Big[\frac{p^{\prime \prime }(x)}{6\sigma \Delta u_{L}}(\gamma \sqrt{\sigma\Delta u_{L}}-Z)^{3}+\frac{Ap(x)}{24\Delta ^{2}\sigma ^{2}u_{L}}Z^{4}(\gamma
\sqrt{\sigma \Delta u_{L}}-Z)\left\vert Z\leq \zeta _{L}\right. \Big ]}{%
p(x)E(\gamma \sqrt{\sigma\Delta u_{L}}-Z\vert Z\leq \zeta _{L} )}+o(u_L^{-1}+ yu_L^{-1})\right\} .
\end{eqnarray*}
We insert the above identity back to the expression of $v'(x)$
and obtain that
\begin{eqnarray*}
v'(x) &=&\exp \Big\{\lambda (u_{L})+ o( yu_L^{-1}) +O(y^2zu_L^{-2})+\omega (u_{L})+\sigma u_{L}+\sigma w+\frac{\sigma
y^{2}}{2\Delta u_{L}}+\frac{A\sigma
u_{L}}{24}\gamma ^{4}\Big\} \\
&&\times\frac{1}{\sqrt{\Delta \sigma u_{L}}}\exp \Big\{\frac{z}{2u_{L}}-\frac{A}{24\Delta
^{2}\sigma u_{L}}E( Z^{4}|Z\leq \zeta _{L}) \Big\} \notag\\
&&\times H_{L,x}\Big(\gamma \sqrt{\sigma \Delta (u_{L}-z)},\sqrt{1-\frac{z}{u_{L}}}\zeta _{L}-\sqrt{\frac{\sigma }{\Delta (u_{L}-z)}}y; u_L \Big)\notag \\
&&\times\exp \left\{ \frac{E\Big[\frac{p^{\prime \prime }(x)}{6\sigma \Delta u_{L}}(\gamma \sqrt{\sigma\Delta u_{L}}-Z)^{3}+\frac{Ap(x)}{24\Delta ^{2}\sigma ^{2}u_{L}}Z^{4}(\gamma
\sqrt{\sigma \Delta u_{L}}-Z)\left\vert Z\leq \zeta _{L}\right. \Big ]}{%
p(x)E(\gamma \sqrt{\sigma\Delta u_{L}}-Z\vert Z\leq \zeta _{L} )}\right\} ,\notag
\end{eqnarray*}%
where
$$
H_{L,y}(x,\zeta; u)\triangleq e^{-\frac{x^{2}}{2}} \times
E\Big[p(y)(x-Z)- \frac{p^{\prime}(y)}{2\sqrt{\Delta \sigma u}}(x-Z)^{2}   ~\Big | ~
Z\leq \zeta \Big].
$$
\end{proof}

\bigskip

\begin{proof}[Proof of Lemma \ref{LemA}]
We insert $\gamma _{L}=\frac{\zeta _{L}}{\sqrt{\sigma \Delta u_{L}}}$ to the expression of $\mathcal{A}$ in \eqref{AA} and obtain that
\begin{eqnarray*}
\mathcal{A} &=&\lambda (u_{L})+ o( yu_L^{-1}) +O(y^2zu_L^{-2})+\sigma w+\frac{\sigma y^{2}}{2\Delta u_{L}}
+\frac{A\zeta _{L}^{4}}{24\Delta ^{2}\sigma u_{L}}+\frac{z}{2u_{L}}-\frac{AE\left( Z^{4}|Z\leq \zeta _{L}\right) }{24\Delta ^{2}\sigma u_{L}}\notag \\
&&+ G_{L}\Big(\sqrt{1-\frac{z}{u_{L}}}\zeta _{L}-\sqrt{\frac{\sigma }{\Delta (u_{L}-z)}}y;u_L\Big ) - G_{L}(\zeta_L;u_L) \notag\\
&&+\frac{E[\frac{p^{\prime \prime }(L)}{6\sigma \Delta u_{L}}(\zeta _{L}-Z)^{3}+\frac{Ap(L)}{24\Delta ^{2}\sigma
^{2}u_{L}}Z^{4}(\zeta _{L}-Z)\left\vert Z\leq \zeta _{L}\right. ]}{%
p(L)E(\zeta _{L}-Z\left\vert Z\leq \zeta _{L}\right. )}.
\end{eqnarray*}%
Note that   $
\Xi _{L}=-\lim_{u_{L}\rightarrow \infty }\partial^2_\zeta G_{L}(\zeta_{L}, u_L)$. Then,%
\begin{eqnarray*}
\mathcal{A} &\mathcal{=}&\lambda (u_{L})+ o( yu_L^{-1}) +O(y^2zu_L^{-2})+\sigma w+\frac{\sigma y^{2}}{2\Delta u_{L}}+\frac{A\zeta _{L}^{4}}{24\Delta ^{2}\sigma u_{L}}+\frac{z}{2u_{L}}-\frac{AE\left( Z^{4}|Z\leq \zeta _{L}\right) }{24\Delta ^{2}\sigma u_{L}} \\
	&&+\frac{E[\frac{p^{\prime \prime }(L)}{6\sigma \Delta u_{L}}(\zeta
_{L}-Z)^{3}+\frac{Ap(L)}{24\Delta ^{2}\sigma ^{2}u_{L}}Z^{4}(\zeta
_{L}-Z)\left\vert Z\leq \zeta _{L}\right. ]}{p(L)E(\zeta _{L}-Z\vert Z\leq \zeta _{L})} \\
&&
-\frac{\Xi _{L}+o(1)}{2}\Big(\frac{\zeta _{L}z}{2u_{L}}+\sqrt{\frac{%
\sigma }{\Delta (u_{L}-z)}}y\Big)^{2}. \\
&=&\lambda (u_{L})+ o( yu_L^{-1}) +O(y^2zu_L^{-2})+\sigma w+\frac{\sigma y^{2}}{2\Delta u_{L}}+\frac{z}{2u_{L}}+\frac{\kappa_L }{u_{L}}
-\frac{\Xi _{L}+o(1)}{2}\Big(\frac{\zeta _{L}z}{2u_{L}}+\sqrt{\frac{%
\sigma }{\Delta (u_{L}-z)}}y\Big)^{2}.
\end{eqnarray*}%
where $\kappa_L$ is given as in \eqref{kappa}.
\end{proof}

\bigskip

\begin{proof}[Proof of Lemma \ref{LemMinor}]
In this case that $\vert \sqrt{1-\frac{z}{u_{L}}}\zeta _{L}-\sqrt{\frac{\sigma }{\Delta (u_{L}-z)}}y - \zeta_L\vert> \varepsilon $, the maximum of $|v'(x)|$ is not necessarily attained at $x=L$. Note that this does not change the calculation very much except that the terms $p(x)$ and $p'(x)$ in $H_{x,L}$ may not be evaluated on the boundary $x=L$, but still in the region $[L- u^{-1/2 +\delta}, L]$. Therefore, maximizing \eqref{HH}, we have that
\begin{eqnarray*}
&&\sup_{x\in [L- u^{-1/2 +\delta}, L]}\log\Big| H_{L,x}(\gamma \sqrt{\sigma \Delta (u_{L}-z)},\sqrt{1-\frac{z}{u_{L}}}\zeta _{L}-\sqrt{\frac{\sigma }{\Delta (u_{L}-z)}}y;u_L)\Big|\\
& =& G_{L}\Big(\sqrt{1-\frac{z}{u_{L}}}\zeta _{L}-\sqrt{\frac{\sigma }{\Delta (u_{L}-z)}}y;u_L\Big )  + O(u^{-1/2 +\delta}).
\end{eqnarray*}
Therefore, we only need to add an $O(u^{-1/2+\delta})$ to the definition of $\mathcal A$ in \eqref{AA}.
Furthermore, the term in \eqref{AA}
is bounded by
$$G_{L}\Big(\sqrt{1-\frac{z}{u_{L}}}\zeta _{L}-\sqrt{\frac{\sigma }{\Delta (u_{L}-z)}}y,u_L\Big ) - G_{L}(\zeta_L,u_L)\leq - \delta_0\varepsilon ^2$$
for some $\delta_0 >0$.
Furthermore, on the set $\mathcal L^*_{u}$ we have that
$\lambda (u_{L})+ o( yu_L^{-1}) +O(y^2zu_L^{-2})=o(1).$
Therefore, we have the bound
$S(w,y,z)\geq u_{L}^{2}+w^{2}+\frac{\Delta ^{2}(w+z)^{2}}{A-\Delta ^{2}} +2u_{L}\mathcal{A}/\sigma + \delta_0 \varepsilon^2 u_L$
and further
\begin{eqnarray*}
P\left( \max_{x\in \lbrack L-u_{L}^{-1/2+\delta },L]}|v^{\prime
}(x)|>b; \mathcal L_{u_L}^*; ~~\left\vert \sqrt{1-\frac{z}{u_{L}}}\zeta _{L}-\sqrt{\frac{\sigma }{\Delta (u_{L}-z)}}y - \zeta_L\right\vert\geq \varepsilon \right) &= &o(1)u_{L}^{-1}e^{-u_{L}^{2}/2}.
\end{eqnarray*}
\end{proof}

\end{document}